\documentclass[11pt,reqno,oneside]{amsart}
\usepackage[usenames,dvipsnames,svgnames,table]{xcolor}
\usepackage{graphicx,verbatim}
\usepackage{amssymb,cite,bm,soul}
\usepackage{epstopdf}
\usepackage{graphicx}
\usepackage{amsmath}
\usepackage{amsthm, enumerate}
\usepackage{mathrsfs}
\usepackage{caption}
\usepackage{subcaption}
\usepackage[normalem]{ulem}   
\allowdisplaybreaks
\chardef\forshowkeys=0
\chardef\showllabel=0
\chardef\refcheck=0
\chardef\sketches=0
\usepackage{float}
\usepackage[shortlabels]{enumitem}

\floatstyle{boxed}

\usepackage{marginnote}

\usepackage[colorlinks=true, pdfstartview=FitV, linkcolor=blue, citecolor=blue, urlcolor=blue]{hyperref}

\ifnum\forshowkeys=1

\usepackage[notref,notcite,color]{showkeys}
\fi

\author[M.~Aydin]{Mustafa Sencer Aydin}
\address[Aydin]{\newline
	Department of Mathematics, University of Southern California, Los Angeles, CA 90089}
\email[]{\href{maydin@usc.edu}{maydin@usc.edu}}

\author[Jayanti]{Pranava Chaitanya Jayanti}
\address[Jayanti]{\newline
	Department of Mathematics, University of Southern California, Los Angeles, CA 90089, USA}
\email[]{\href{pjayanti@usc.edu}{pjayanti@usc.edu}}

\title[Fractional regularity, global persistence, and asymptotics for the Boussinesq equations]{Fractional regularity, global persistence, and asymptotic properties of the Boussinesq equations on bounded domains}

\usepackage{enumitem}
\usepackage{datetime}
\usepackage{fancyhdr}
\usepackage{comment}
\allowdisplaybreaks
\usepackage[margin=1in]{geometry}
\usepackage{amsmath, amsthm, amssymb}
\usepackage{times}
\usepackage{graphicx}

\ifnum\refcheck=1
\usepackage{refcheck}
\fi

\begin{document}
	
	\def\inprogress{{\colg IN PROGRESS} }
	\def\bnew{\colr {} NEW: }
	\def\enew{\colb {}}
	\def\bold{\colu {} OLD: }
	\def\eold{\colb{}}
	\def\YY{X}
	\def\OO{\mathcal O}
	\def\SS{\mathbb S}
	\def\CC{\mathbb C}
	\def\RR{\mathbb R}
	\def\TT{\mathbb T}
	\def\ZZ{\mathbb Z}
	\def\HH{\mathbb H}
	\def\RSZ{\mathcal R}
	\def\LL{\mathcal L}
	\def\SL{\LL^1}
	\def\ZL{\LL^\infty}
	\def\GG{\mathcal G}
	\def\tt{\langle t\rangle}
	\def\erf{\mathrm{Erf}}
	\def\veps{\varepsilon}
	\def\mgt#1{\textcolor{magenta}{#1}}
	\def\ff{\rho}
	\def\gg{G}
	\def\sqrtnu{\sqrt{\nu}}
	\def\ww{w}
	\def\ft#1{#1_\xi}
	\def\ges{\gtrsim}
	\renewcommand*{\Re}{\ensuremath{\mathrm{{\mathbb R}e\,}}}
	\renewcommand*{\Im}{\ensuremath{\mathrm{{\mathbb I}m\,}}}
	\ifnum\showllabel=0
	\def\llabel#1{\marginnote{\color{gray}\rm(#1)}[-0.0cm]\notag}
	\else
	\def\llabel#1{\notag}
	\fi
	\newcommand{\norm}[1]{\left\|#1\right\|}
	\newcommand{\nnorm}[1]{\lVert #1\rVert}
	\newcommand{\abs}[1]{\left|#1\right|}
	\newcommand{\NORM}[1]{|\!|\!| #1|\!|\!|}
	\newtheorem{theorem}{Theorem}[section]
	\newtheorem{Theorem}{Theorem}[section]
	\newtheorem{corollary}[theorem]{Corollary}
	\newtheorem{Corollary}[theorem]{Corollary}
	\newtheorem{proposition}[theorem]{Proposition}
	\newtheorem{Proposition}[theorem]{Proposition}
	\newtheorem{Lemma}[theorem]{Lemma}
	\newtheorem{lemma}[theorem]{Lemma}
	\theoremstyle{definition}
	\newtheorem{definition}{Definition}[section]
	\newtheorem{remark}[Theorem]{Remark}
	\def\theequation{\thesection.\arabic{equation}}
	\numberwithin{equation}{section}
	\definecolor{mygray}{rgb}{.6,.6,.6}
	\definecolor{myblue}{rgb}{9, 0, 1}
	\definecolor{colorforkeys}{rgb}{1.0,0.0,0.0}
	\newlength\mytemplen
	\newsavebox\mytempbox
	\makeatletter
	\newcommand\mybluebox{%
		\@ifnextchar[
		{\@mybluebox}%
		{\@mybluebox[0pt]}}
	\def\@mybluebox[#1]{%
		\@ifnextchar[
		{\@@mybluebox[#1]}%
		{\@@mybluebox[#1][0pt]}}
	\def\@@mybluebox[#1][#2]#3{
		\sbox\mytempbox{#3}%
		\mytemplen\ht\mytempbox
		\advance\mytemplen #1\relax
		\ht\mytempbox\mytemplen
		\mytemplen\dp\mytempbox
		\advance\mytemplen #2\relax
		\dp\mytempbox\mytemplen
		\colorbox{myblue}{\hspace{1em}\usebox{\mytempbox}\hspace{1em}}}
	\makeatother
	\def\XX{{\mathcal X}}
	\def\XXT{{\mathcal X}_T}
	\def\XXTzero{{\mathcal X}_{T_0}}
	\def\XXi{{\mathcal X}_\infty}
	\def\YY{{\mathcal Y}}
	\def\YYT{{\mathcal Y}_T}
	\def\YYTzero{{\mathcal Y}_{T_0}}
	\def\YYi{{\mathcal Y}_\infty}
	\def\cc{\text{c}}
	\def\rr{r}
	\def\weaks{\text{\,\,\,\,\,\,weakly-* in }}
	\def\inn{\text{\,\,\,\,\,\,in }}
	\def\cof{\mathop{\rm cof\,}\nolimits}
	\def\Dn{\frac{\partial}{\partial N}}
	\def\Dnn#1{\frac{\partial #1}{\partial N}}
	\def\tdb{\tilde{b}}
	\def\tda{b}
	\def\qqq{u}
	\def\lat{\Delta_2}
	\def\biglinem{\vskip0.5truecm\partialar==========================\partialar\vskip0.5truecm}
	\def\inon#1{\hbox{\ \ \ \ \ \ \ }\hbox{#1}}                
	\def\onon#1{\inon{on~$#1$}}
	\def\inin#1{\inon{in~$#1$}}
	\def\FF{F}
	\def\andand{\text{\indeq and\indeq}}
	\def\ww{w(y)}
	\def\ll{{\color{red}\ell}}
	\def\ee{\epsilon_0}
	\def\Leray{\mathbb{P}}
	\def\startnewsection#1#2{ \section{#1}\label{#2}\setcounter{equation}{0}}   
	\def\loc{\text{loc}}
	\def\nnewpage{ }
	\def\sgn{\mathop{\rm sgn\,}\nolimits}    
	\def\Tr{\mathop{\rm Tr}\nolimits}    
	\def\div{\mathop{\rm div}\nolimits}
	\def\curl{\mathop{\rm curl}\nolimits}
	\def\dist{\mathop{\rm dist}\nolimits}  
	\def\supp{\mathop{\rm supp}\nolimits}
	\def\indeq{\quad{}}           
	\def\partialeriod{.}                       
	\def\semicolon{\,;}                  
	\def\colr{\color{red}}
	\def\colrr{\color{black}}
	\def\colb{\color{black}}
	\def\coly{\color{lightgray}}
	\definecolor{colorgggg}{rgb}{0.1,0.5,0.3}
	\definecolor{colorllll}{rgb}{0.0,0.7,0.0}
	\definecolor{colorhhhh}{rgb}{0.3,0.75,0.4}
	\definecolor{colorpppp}{rgb}{0.7,0.0,0.2}
	\definecolor{coloroooo}{rgb}{0.45,0.0,0.0}
	\definecolor{colorqqqq}{rgb}{0.1,0.7,0}
	\def\colg{\color{colorgggg}}
	\def\collg{\color{colorllll}}
	\def\cole{\color{colorpppp}}
	\def\cole{\color{coloroooo}}
	\def\coleo{\color{colorpppp}}
	\def\colu{\color{blue}}
	\def\colc{\color{colorhhhh}}
	\def\colW{\colb}   
	\definecolor{coloraaaa}{rgb}{0.6,0.6,0.6}
	\def\colw{\color{coloraaaa}}
	\def\comma{ {\rm ,\qquad{}} }            
	\def\commaone{ {\rm ,\quad{}} }          
	\def\lec{\lesssim}
	\def\nts#1{{\color{red}\hbox{\bf ~#1~}}} 
	\def\ntsf#1{\footnote{\color{colorgggg}\hbox{#1}}}
	\def\ntsik#1{{\color{purple}\hbox{\bf ~#1~}}} 
	\def\blackdot{{\color{red}{\hskip-.0truecm\rule[-1mm]{4mm}{4mm}\hskip.2truecm}}\hskip-.3truecm}
	\def\bluedot{{\color{blue}{\hskip-.0truecm\rule[-1mm]{4mm}{4mm}\hskip.2truecm}}\hskip-.3truecm}
	\def\partialurpledot{{\color{colorpppp}{\hskip-.0truecm\rule[-1mm]{4mm}{4mm}\hskip.2truecm}}\hskip-.3truecm}
	\def\greendot{{\color{colorgggg}{\hskip-.0truecm\rule[-1mm]{4mm}{4mm}\hskip.2truecm}}\hskip-.3truecm}
	\def\cyandot{{\color{cyan}{\hskip-.0truecm\rule[-1mm]{4mm}{4mm}\hskip.2truecm}}\hskip-.3truecm}
	\def\reddot{{\color{red}{\hskip-.0truecm\rule[-1mm]{4mm}{4mm}\hskip.2truecm}}\hskip-.3truecm}
	\def\gdot{\greendot}
	\def\tdot{\gdot}
	\def\bdot{\bluedot}
	\def\ydot{\cyandot}
	\def\rdot{\reddot}
	\def\fractext#1#2{{#1}/{#2}}
	\def\ii{\hat\imath}
	\def\fei#1{\textcolor{blue}{#1}}
	\def\vlad#1{\textcolor{cyan}{#1}}
	\def\igor#1{\text{{\textcolor{colorqqqq}{#1}}}}
	\def\igorf#1{\footnote{\text{{\textcolor{colorqqqq}{#1}}}}}
	\newcommand{\UE}{U^{\rm E}}
	\newcommand{\partialE}{P^{\rm E}}
	\newcommand{\KP}{K_{\rm P}}
	\newcommand{\uNS}{u^{\rm NS}}
	\newcommand{\vNS}{v^{\rm NS}}
	\newcommand{\partialNS}{p^{\rm NS}}
	\newcommand{\omegaNS}{\omega^{\rm NS}}
	\newcommand{\uE}{u^{\rm E}}
	\newcommand{\vE}{v^{\rm E}}
	\newcommand{\omegaE}{\omega^{\rm E}}
	\newcommand{\ua}{u_{\rm   a}}
	\newcommand{\va}{v_{\rm   a}}
	\newcommand{\omegaa}{\omega_{\rm   a}}
	\newcommand{\ue}{u_{\rm   e}}
	\newcommand{\ve}{v_{\rm   e}}
	\newcommand{\omegae}{\omega_{\rm e}}
	\newcommand{\omegaeic}{\omega_{{\rm e}0}}
	\newcommand{\ueic}{u_{{\rm   e}0}}
	\newcommand{\veic}{v_{{\rm   e}0}}
	\newcommand{\up}{u^{\rm P}}
	\newcommand{\vp}{v^{\rm P}}
	\newcommand{\N}{\mathbb{N}}
	\newcommand{\tup}{{\tilde u}^{\rm P}}
	\newcommand{\bvp}{{\bar v}^{\rm P}}
	\newcommand{\omegap}{\omega^{\rm P}}
	\newcommand{\tomegap}{\tilde \omega^{\rm P}}
	\renewcommand{\up}{u^{\rm P}}
	\renewcommand{\vp}{v^{\rm P}}
	\renewcommand{\omegap}{\Omega^{\rm P}}
	\renewcommand{\tomegap}{\omega^{\rm P}}
	\begin{abstract}
		We address the long-time behavior of the 2D Boussinesq system, which consists of the incompressible Navier-Stokes equations driven by a non-diffusive density. We construct globally persistent solutions on a smooth bounded domain, when the initial data belongs to $(H^k\cap V)\times H^k$ for $k\in\N$ and $H^s\times H^s$ for $0<s<2$. The proofs use parabolic maximal regularity and specific compatibility conditions at the initial time. Additionally, we also deduce various asymptotic properties of the velocity and density in the long-time limit and present a necessary and sufficient condition for the convergence to a steady state.
	\end{abstract}
	\maketitle
	\tableofcontents
	
	\section{Introduction} \label{sec:introduction}
	We investigate the well-posedness and persistence of regularity of the two-dimensional incompressible, viscous Boussinesq equations without thermal diffusivity. The governing equations are given by
	\begin{equation} \label{eq:boussinesq equations}
		\begin{aligned}
			u_t - \Delta u + u \cdot \nabla u + \nabla p 
			&= 
			\rho e_2 , \\
			\rho_t + u \cdot \nabla \rho 
			&= 
			0 , \\
			\nabla \cdot u
			&=
			0 ,
		\end{aligned}
	\end{equation}
	subjected to the boundary and initial conditions
	\begin{equation} \label{eq:boundary and initial conditions}
		u |_{\partial\Omega} = 0 \comma (u,\rho)(0) = (u_0, \rho _0) ,
	\end{equation}
	where $\Omega \subseteq \mathbb{R}^2 $ is a smooth, open, and bounded domain. The Boussinesq equations in various forms have been extensively studied in physics and mathematics, finding applications in various fields such as oceanography, biophysics, and atmospheric turbulence modeling. Fundamentally, they are derived using the ``Boussinesq approximation,'' which posits that density variations have negligible effects on the velocity field except through buoyancy. This reduces the degree of nonlinearity of the system, leading to the incompressible Navier-Stokes with a (linear) density-dependent force that acts in the vertical direction alone. Since density variations in geophysical fluid dynamics are often a result of thermal fluctuations, it is also common to replace the density $\rho$ by the temperature $\theta$. The two fields are equivalent insofar as the Boussinesq approximation goes, and this is the reason for the label ``thermal diffusivity'' for the term $\kappa\Delta\rho$ that is sometimes employed in the continuity equation. Apart from physical applications, the Boussinesq equations have also generated much theoretical interest for their connections with the 3D Euler equations and the phenomenon of vortex stretching~\cite{DWZZ}.
	
	In the presence of a positive viscosity and a positive thermal diffusivity (in one or both directions), there are well-established results regarding the global existence of solutions~\cite{ACW,BrS,CW} and also blow-up criteria~\cite{QDY}. At the other end of the spectrum, the questions of global well-posedness for the system \eqref{eq:boussinesq equations} in the inviscid and non-diffusive case is partially unresolved. Progress in this direction includes weak solutions, local existence, finite time singularities, and blow-up criteria; see~\cite{BS,CH1,CH2,CKN,CN,EJ,HKZ2}. The most popular class of models in the literature correspond to the ``viscous and non-diffusive'' case. Among the first results in this regard (on the whole space $\mathbb{R}^2$) were due to Chae \cite{C} and Hou and Li \cite{HL}, who obtained the global existence and persistence of the regularity with $H^{k} \times H^{k-1}$ initial data with $k = 3, 4, \ldots$. For $k=2$, \cite{HKZ1} established the persistence of the regularity for Dirichlet and periodic boundary conditions. In a recent work, \cite{AKZ} proved the global well-posedness of the system with initial data in $D(A)\times H^1$, ensuring that the regularity of solutions persists over time with several additional asymptotic properties. The case of initial data in $H\times L^2$ was addressed in~\cite[Section 5]{H}, while some higher regularity results were established in~\cite{Ju}. Other interesting works involving strong solutions for the viscous and non-diffusive case include~\cite{DWZZ,HK1,KW1,KW2,KWZ,LLT}. In~\cite{DP}, Danchin and Paicu established global-in-time weak solutions for dimensions $N\ge 2$, in both Sobolev and Besov spaces. Another interesting limiting case is that of no viscosity, and yet a positive diffusivity~\cite{HK2,HKR,HKZ2}. The study of the Boussinesq equations has taken other paths too: global attractors~\cite{BFL}, inviscid/non-diffusive limits~\cite{BS}, critical fractional dissipation~\cite{HS,JMWZ,SW}, temperature dependent viscosity and diffusivity~\cite{WZ}, Navier boundary conditions~\cite{HW,HWW+}, stability analysis~\cite{JK,KLN}, norm growth~\cite{KPY}, thin domains~\cite{S}, and several others~\cite{CD,CG,J,KTW}. While most of the aforementioned works are in 2D, some involve more dimensions~\cite{BS,BrS,DP,QDY,W,WY}. For a comprehensive review of the results pertaining to the Boussinesq equations, see~\cite{ACSetal,KLN}.
	
	In this work, we provide an affirmative answer to the question of existence and uniqueness for $(H^k\cap V) \times H^k$ initial data (Theorem~\ref{T01}). This result can be compared to~\cite{LPZ} where the authors used $H^3\times H^3$ data to show global strong solutions (under certain compatibility conditions on the spatial regularity of the initial data), while our approach imposes time regularity on the solution at $t=0$. In the case of zero thermal diffusivity, other attempts~\cite{C,HL} have often worked with initial velocity with (one derivative) higher regularity than the initial density. We are able to overcome this gap, by showing that it is sufficient for the initial velocity and density to have identical regularity. It must be noted that the results for $k=1,2$ were already obtained in~\cite{Ju}. However, to the best of our knowledge, the corresponding results for arbitrary $k\in\N$ (on bounded domains) has remained open until now.
	
	The main challenge, when dealing with $H^1$ initial data, is the continuity equation. Since the density evolves via transport-type dynamics, the existence of a strong solution requires that the $W^{1,\infty}$ norm of the velocity be integrable in time. To achieve this, \cite{AKZ} use the maximal parabolic regularity of the Stokes equation with the initial data in $D(A)$. Here, we show that the same technique works even when the initial data is only marginally more regular than $L^2$ (Lemma~\ref{lem:unique particle trajectories}). The approach, inspired by~\cite{DM}, involves trading space integrability of the velocity for time integrability. Furthermore, by adapting the techniques employed in~\cite{AKZ}, we present a concise proof of the persistence of regularity in the $D(A)\times H^2$ space. Subsequently, we generalize our result to arbitrarily high regularity, demonstrating the persistence of the solution in $H^k\times H^k$. Certain compatibility criteria on the time-regularity of the solution at $t=0$ are necessary (see~\cite{T3}), which motivates the functional setting we employ. An induction argument, combined with the preceding results for $D(A)\times H^2$, yields the sought after solution.

    In this paper, we also tackle the situation of global persistence when the initial data has a regularity of fractional order (Theorem~\ref{T02}) . Specifically, for $0<s<2$, we consider the initial data $(u_0,\rho_0)\in D(A^{\frac{s}{2}})\times H^r$, where $A$ is the Stokes operator and $r = \lfloor s \rfloor$ or $s$. The lack of dissipation in the continuity equation poses a problem, but it is handled via extension to $\RR^2$ followed by commutator estimates to handle the highest order derivatives. The case when $r=\lfloor s \rfloor$ serves as a fractional analogue to the initial data in $H^k\times H^{k-1}$, i.e., when the density has lower regularity than the velocity.

    Finally, for initial data in $D(A)\times H^2$, we establish some asymptotic properties of the solution, as $t\to\infty$ (Theorem~\ref{T05}). In particular, we show that $Au - \Leray(\rho e_2)$ vanishes in the $H^1$ norm as $t\to\infty$ (where $\Leray$ is the Leray projector), as does the curl-free, trace-free part of $\nabla p-\rho e_2$. The former is an improvement to an analogous result in~\cite{AKZ}, since we have more regular density. Additionally, we also present a necessary and sufficient condition for $\rho$ to converge to a steady state, which is identified to be that of a stratified fluid. If and when said steady state is achieved, we show that there only exists a vertical pressure gradient, which exactly balances out the force due to the stratified density field. Moreover, in such a situation, the velocity also asymptotically vanishes in the $H^2$ norm, which confirms the intuition that the system $(u,\rho)$ would eventually down to $(0,\overline{\rho})$, where $\overline{\rho}$ does not depend on the horizontal coordinate.

	\section{Preliminaries and the main results} \label{sec:prelim and main results}
	
	We now define the function spaces and the notation that appear throughout this paper, before presenting the main results. With $\Omega \subset \mathbb{R}^2$, a smooth bounded domain, we introduce the spaces
	\begin{equation} \label{eq:defining H and V}
		\begin{aligned}
			& H 
			= 
			\{u \in L^2(\Omega) : \nabla \cdot u = 0 \text{ in } \Omega, u \cdot n = 0 \text{ on } \partial \Omega\}, \\
			&V 
			= 
			\{u \in H^1_0 (\Omega) : \nabla \cdot u = 0 \text{ in } \Omega\} ,
		\end{aligned}
	\end{equation}
	where $n$ represents the outward unit normal vector ~\cite{CF,T1,T2}. Additionally, we define the Stokes operator
	\begin{equation} \label{eq:defining Stokes operator}
		A 
		= 
		-\mathbb{P} \Delta ,
	\end{equation}
	where $D(A) =H^2(\Omega) \cap V$ denotes its domain and 
	$\mathbb{P}\colon L^2 \to H$ refers to the Leray projector. In what follows, for a fixed $T>0$, $L^p V$ is understood as $L^p([0,T],V)$, $C V$ as $C([0,T],V)$, and so on. Applying the Leray projector $\mathbb{P}$ to the first equation in \eqref{eq:boussinesq equations},
	we may express the momentum equation as
	\begin{equation} \label{eq:Leray projected momentum equation}
		u_t + Au + \mathbb{P}(u \cdot \nabla u)
		= 
		\mathbb{P}(\rho e_2) ,
	\end{equation}
	which represents the usual equivalent formulation for the first and third equations in ~\eqref{eq:boussinesq equations}. In ~\cite{AKZ}, the following regularity theorem was proven for the Boussinesq system.
	
	\begin{Theorem}[{\cite[Theorem 2.1]{AKZ}} Global persistence with $D(A)\times H^1$ data]
		\label{T00}
		Assume that $(u_0,\rho_0) \in D(A) \times H^{1}$. Then, for any $T\in (0,\infty)$,
		\begin{enumerate}
			\item the Boussinesq system \eqref{eq:boussinesq equations} has a unique solution with the regularity
			\begin{equation}
				\begin{gathered}
					u 
					\in 
					CV \cap L^1 W^{1,\infty} \cap L^2 D(A) \cap L^3 W^{2,3} \cap W^{1,\infty} L^2 \cap H^2 V' , \\
					\rho 
					\in 
					C H^1 \cap H^1 L^2 ,
				\end{gathered}
			\end{equation}
			\item the solution satisfies
			\begin{equation} \label{eq:T00 auxiliary estimates for velocity}
				\begin{gathered}
					\lim_{t\rightarrow\infty} \norm{\nabla u(t)}_{L^2} 
					= 
					0 , \\
					\lim_{t\rightarrow\infty} \norm{A u(t) - \Leray(\rho(t) e_2)}_{L^2} 
					= 
					0 , \\
					\norm{Au(t)}_{L^2} 
					\le 
					C ,
				\end{gathered}
			\end{equation}
			and for every $\epsilon>0$, there exists $C_{\epsilon}>0$ such that
			\begin{equation}
				\norm{\rho(t)}_{H^1} 
				\le 
				C_{\epsilon}e^{\epsilon t} ,
			\end{equation}
			where $C$ and $C_{\epsilon}$ depend on the size of the initial data, and
			\item the function $\Leray(\rho(t)e_2)$ weakly converges to $0$ in $H$, as $t\rightarrow\infty$.
		\end{enumerate}
	\end{Theorem}
	
	In this paper, we establish several other global persistence results, beginning from different classes of initial conditions. The first one among them considers the case when $(u_0,\rho_0)\in (H^k\cap V)\times H^k$, for $k\in\N$.

	\begin{Theorem}[Global persistence with $(H^k\cap V)\times H^k$ data] \label{T01}  
        Given any $T\in (0,\infty)$ and initial conditions $(u_0,\rho_0)$, the Boussinesq system~\eqref{eq:boussinesq equations} has a unique global-in-time solution $(u,\rho)$ on $[0,T]$ for all of the following cases: 
        \begin{enumerate}[(i)]
		    \item $k=1$: $(u_0,\rho_0) \in V\times H^1$. Then, 
            \begin{equation} \label{eq:solution regularity for k=1}
            \begin{gathered}
                (u,\rho)\in (CV\cap L^1 W^{1,\infty}\cap L^{\infty}D(A)) \times CH^1 , \\
                (u_t, \rho_t) \in CV' \times CH^{-1} .
            \end{gathered}
            \end{equation}
            Furthermore, all of the asymptotic properties from Theorem~\ref{T00} remain valid here (except for the boundedness of $\Vert Au(t)\Vert_{L^2}$, which only holds for $t>0$).

            \item $k=2$: $(u_0,\rho_0) \in D(A)\times H^2$. Then, 
            \begin{equation} \label{eq:solution regularity for k=2}
            \begin{gathered}
                (u,\rho)\in (CD(A)\cap L^1 W^{1,\infty}) \times CH^2 , \\
                (u_t, \rho_t) \in CH \times CL^2 .
            \end{gathered}
            \end{equation}

            \item $k\ge 3$: $(u_0,\rho_0) \in (H^k\cap V)\times H^k$, and the compatibility conditions in~\eqref{eq:compatibility conditions for data} hold. Then, 
            \begin{equation} \label{eq:solution regularity for k>=3}
                (u,\rho)\in W_k \times \tilde{W}_k \subseteq CH^k \times CH^k ,
            \end{equation}
            where the spaces $W_k$ and $\tilde{W}_k$ are defined in~\eqref{eq:defining W_k and Tilde W_k}.
		\end{enumerate}
	\end{Theorem}
	
	Some remarks about these results are in order.
    \begin{remark}
        We emphasize the persistence aspect of the theorem; namely, if the initial datum $(u_0,\rho_0)$ belongs to $(H^k\cap V)\times H^k$, then the solution $(u(t),\rho(t))$ belongs to $(H^k\cap V)\times H^k$ for all $t\ge 0$.
    \end{remark}
    \begin{remark}
        It follows from the proof of~\eqref{eq:solution regularity for k=1} that the velocity $u(t)$ belongs to $D(A)$ for almost every $t>0$. In fact, $u \in L^2_{\loc}(0,\infty;D(A))$, meaning that we may apply Theorem ~\ref{T00} to deduce that our solution belongs to a finer space.
    \end{remark}
    \begin{remark}
        Combining Theorem~\ref{T01}(i) with Theorem~\ref{T00}, under the assumption that the initial data belongs to $D(A)\times H^2$, essentially proves part (ii). It should be noted that the improvement here is the $H^2$ persistence of the density. In fact, Theorem~\ref{T00} shows that the solution stays in $D(A)\times H^1$, if it emanates from an initial data belonging to the same space.
    \end{remark}
    \begin{remark}
        In Theorem~\ref{T01}(iii), we require the (time derivatives of the) initial conditions to satisfy some compatibility criteria discussed in Section~\ref{sec:proof of H^k x H^k persistence}. This added restriction actually returns in the form of increased regularity of the solution in terms of the spaces $W_k$ and $\tilde{W}_k$. To the best of our knowledge, this result is entirely new for the case of bounded domains.
    \end{remark}
    \begin{remark}
        Parts (i) and (ii) of Theorem~\ref{T01} have been proven in~\cite[Sections 4 and 5]{Ju}. We still present them here for two reasons. Firstly, they are used as the base cases for an induction argument for part (iii), so it makes our proof self-contained. Secondly, the approach used in~\cite{Ju} is different from ours: they employ a spectral decomposition similar to Littlewood-Paley, whereas we use the technique of parabolic maximal regularity.
    \end{remark}
	
	To prove Theorem~\ref{T01}, we employ a Galerkin scheme and derive a priori estimates for the solutions to an approximate system. We may then pass to the limit of the approximation to obtain the desired solutions. One of the key steps involves controlling the quantity $\Vert \nabla u\Vert_{L^1 L^{\infty}}$. In~\cite{AKZ}, this is achieved via the $L^{3}W^{2,3}$ estimate on the velocity, which however, requires the initial velocity to be more regular than~$H^1$. Here, we significantly weaken this restriction to just requiring that $u_0\in H^s$ for any $s>0$. This is done by trading the $L^{3}W^{2,3}$ estimate for the weaker $L^{1+\delta}W^{2,2+\veps}$ norm (for appropriate choices of $\delta,\veps>0$), and taking advantage of the smoothing effect of the Navier-Stokes equation. As a consequence of this maximal parabolic regularity, we are able to control $u$ in certain Sobolev spaces that are continuously embedded into~$L^1 W^{1,\infty}$. 
 
    Next, we present a global persistence theorem for spaces of fractional Sobolev regularity. We note that for any $s\in (0,2)$ we have $D(A^{\frac{s}{2}}) = H_0^s\cap H$; see~\cite{FHR}. Then, our precise result is as follows.

	\begin{Theorem}[Global persistence with fractional regularity data] \label{T02}  
        Given $(u_0,\rho_0)\in D(A^{\frac{s}{2}})\times H^r$, where $r=\lfloor s\rfloor$ or $s$, the Boussinesq system~\eqref{eq:boussinesq equations} has a unique global-in-time solution satisfying 
        \begin{equation} \label{eq:solution regularity for lower fractional full}
            \begin{gathered}
                (u,\rho)\in (CD(A^{\frac{s}{2}})\cap L^1W^{1,\infty}\cap L^{\infty}D(A^{\frac{s+1}{2}})) \times C H^r , \\
                (u_t, \rho_t) \in CD(A^{\frac{s-1}{2}}) \times CH^{r-1} .
            \end{gathered}
        \end{equation}
        The case $s\in (0,1)$ is referred to as ``lower fractional regularity,'' while the case $s\in (1,2)$ is referred to as the ``higher fractional regularity''.
	\end{Theorem}

    The proof of Theorem~\ref{T02} is achieved in two parts. Beginning from the approximate system for the velocity, its estimates can be derived without needing $r=s$, and we identify a subsequence that converges weakly to the desired solution $u$. Indeed, it is sufficient to have $r=\lfloor s\rfloor$, and the corresponding density bounds do not pose any problems. However, when we have to deal with $\rho \in H^s$, we require the use of a commutator estimate. For this reason, when the density has fractional regularity, we extend the solution $u$ and the initial data $(u_0,\rho_0)$ to the whole plane $\RR^2$, and use the continuity of the Sobolev extension to obtain the necessary bounds.

    Finally, we also investigate the asymptotic behavior of the solutions.
    Before stating our result, we define the projection $Q\colon L^2 \to H'$ as
    $Qf = \nabla \psi$, where $\psi$ is the unique solution of
    \begin{align}
    \begin{split}
       \div Qf =\Delta \psi &= \div f \comma \text{ in } \Omega,
    \\
       \psi &= 0 \comma \text{ on } \partial \Omega.
       \label{eq:projection1}
    \end{split}
    \end{align}
    We remark that $Q$ need not be equal to $I-\mathbb{P}$. The reason is $(I-\mathbb{P})f = \nabla \phi$, where $\phi$ is the unique solution of
    \begin{align}
    \begin{split}
       \Delta \phi &= \div f \comma \text{ in } \Omega,
    \\
       \nabla \phi \cdot n &= f \cdot n \comma \text{ on } \partial \Omega.
       \label{eq:projection2}
    \end{split}
    \end{align}
    In other words, $I-\mathbb{P}$ only extracts the curl-free part of $f$, whereas $Q$ also ensures that it is trace-free. Now, we are in a position to state our result.
    
    \begin{Theorem}[Asymptotic properties] \label{T05}
	Let $(u_0,\rho_0) \in D(A)\times H^2$, and let $(u,\rho)$ be the solution to \eqref{eq:boussinesq equations}. Then, 
    \begin{equation} \label{eq:asy}
        \Vert Au - \mathbb{P}(\rho e_2)\Vert_{V} + \Vert Q(\nabla p - \rho e_2)\Vert_{H^1\cap H'} \to 0 \comma t \to \infty.
    \end{equation}
    Moreover, the following are equivalent:
	\begin{enumerate}[(i)]
        \item $\rho(t)$ converges to a steady state $\bar{\rho}$ in $L^2$ as $t \rightarrow \infty$. 
        
        \item $\lim_{t \to \infty}(I-\mathbb{P})(\rho(t) e_2) = \bar{\rho}e_2$, with $\Vert \bar{\rho}\Vert_{L^2} = \Vert \rho_0\Vert_{L^2}$.
    \end{enumerate}
    When this is the case, we also have
    \begin{equation} \label{eq:asy2}
        \Vert Au \Vert_{L^2} + \Vert \nabla p- \overline{\rho} e_2\Vert_{L^2} \to 0 \comma t \to \infty .
    \end{equation}
    \end{Theorem}
    
    This theorem states that the density achieves a steady state if and only if the curl-free part of $\rho e_2$ converges to a state that has the same $L^2$ norm as $\rho_0$. Thus, the vector $\rho e_2$ must be entirely curl-free in the limit $t\to\infty$, i.e., $\partial_1 \rho$ goes to $0$ asymptotically\footnote{The convergence of the density to the steady state is only in $L^2$ and not in $H^1$. So, when we say that $\nabla\times(\rho e_2)$ vanishes asymptotically, this is understood to be in the distributional sense.}. Physically, this signifies that the steady state $\overline{\rho}$ describes a stratified fluid. We also prove that the $H^2$ norm of the velocity vanishes in the long-time limit. Thus, for initial data in $D(A)\times H^2$, if the system $(u,\rho)$ were to asymptotically converge to any state, then it must be $(0,\overline{\rho})$, where $\overline{\rho}$ only depends on the second coordinate $x_2$. From~\eqref{eq:asy2}, we see that $(\partial_1 p,\partial_2 p) \to (0,\overline{\rho})$ in $L^2$ asymptotically, whence it is evident that the force from the steady state density is balanced by a (purely) vertical pressure gradient.
    
	The paper is organized as follows. In Section~\ref{sec:proof of V x H^1 persistence}, we prove part~(i) of Theorem~\ref{T01}, beginning with the derivation of a priori estimates for an approximation of the Boussinesq system. We also establish the existence of a unique solution to the approximate system given in \eqref{eq:Galerkin approximation}, using a combination of fixed point and uniform boundedness arguments. Passing to the limit gives us a weak solution to the original system, which is subsequently upgraded to strong continuity in time. Next, the required a priori estimates for proving part~(ii) of Theorem~\ref{T01} are addressed in Section~\ref{sec:proof of D(A) x H^2 persistence}. We work with initial data in $(H^k\cap V) \times H^k$ in Section~\ref{sec:proof of H^k x H^k persistence}, and demonstrate the existence of a unique globally persistent solution, i.e., part~(iii) of Theorem~\ref{T01}. Simultaneously, we also specify the higher time-regularity of the solution. In order to achieve this, we use the characterization theorem in \cite{T3} and formulate the associated induction argument in a manner consistent with the coupling between the velocity and the density. The persistence results for fractional regularity (Theorem~\ref{T02}) are established in Section~\ref{sec:proof of fractional regularity}, using a combination of Sobolev extension and commutator estimates. Finally, the asymptotic properties in Theorem~\ref{T05} are proved in Section~\ref{sec:proof of asymptotic properties}.
	
	Before moving on, we state here some results that are used in the course of the proofs of the main results. The first is a maximal parabolic regularity estimate, proven in ~\cite[Theorem 2.8]{GS}.
	\begin{Lemma} \label{lem:maximal regularity for Stokes}
		Consider the time-dependent Stokes problem given by
		\begin{equation}
			\begin{aligned}
				u_t - \Delta u + \nabla p 
				&= 
				f , \\
				\nabla\cdot u 
				&= 
				0 .
			\end{aligned}
		\end{equation}
		For $1<s,q<\infty$, assume that $f\in L^s L^q$, with the initial condition $u(0,x) = a(x) \in D_q^{1-\frac{1}{s},s}$, where $D_q^{1-\frac{1}{s},s}$ is the real interpolation space $(D(A_q),L^q_{\sigma})_{\frac{1}{s},s}$ with $D(A_q)$ denoting the divergence-free subspace of $W^{2,q}$ and $L^q_{\sigma}$ the divergence-free subspace of $L^q$. Then, there exists a unique solution $(u,p)\in (L^s W^{2,q}\cap W^{1,s}L^q \times W^{1,s}L^q)$ satisfying 
		\begin{equation}
			\norm{u_t}_{L^s L^q} + \norm{u}_{L^s W^{2,q}} + \norm{p}_{L^s W^{1,q}} 
			\lesssim 
			\norm{f}_{L^s L^q} + \norm{a}_{D_q^{1-\frac{1}{s},s}}.
		\end{equation}
	\end{Lemma}
	
	The next lemma is an embedding between domains of fractional Stokes operators defined in different $L^p$ spaces. 
	\begin{Lemma} \label{lem:Stokes embedding}
		For the Stokes operator $A_p$ defined in $L^p$, with $\alpha\le\beta$ and $1<p\le q<\infty$ satisfying 
		\begin{equation} \label{eq:numerology for Stokes Sobolev embedding}
			\frac{2\alpha}{n}-\frac{1}{q} 
			\le 
			\frac{2\beta}{n}-\frac{1}{p} ,
		\end{equation}
		we have
		\begin{equation}
			\norm{A^\alpha_q v}_{L^q} 
			\lesssim 
			\norm{A^\beta_pv}_{L^p}. \llabel{EQ280}
		\end{equation}
	\end{Lemma}
	The above result can be found in ~\cite[Section 2]{SvW}. Finally, we present a useful lemma that establishes the existence of unique particle trajectories as soon as $u_0$ is marginally smoother than $L^2$.

    \begin{lemma} \label{lem:unique particle trajectories}
		Let $s>0$ and $p>2$. If $u_0\in H^s$, and $\rho_0\in H^s$ or $L^p$, then there exists unique particle trajectories for the Boussinesq equations. In other words, there exist a unique solution to the ODE
		\begin{equation} \label{eq:Lagrangian trajectories ODE}
			\begin{aligned}
				\frac{d}{dt}\eta(x,t) &= u(\eta(x,t),t) , \\
				\eta(x,0) &= x ,
			\end{aligned}
		\end{equation}
		where $u$ is a solution of~\eqref{eq:boussinesq equations}. 
	\end{lemma}
	
	\begin{proof}
		It is sufficient to consider the case $0<s<1$. From the standard ODE theory, it is evident that $u\in L^1 W^{1,\infty}$ is a sufficient condition for having well-defined particle trajectories. To this end, we first establish $u \in L^{1+\delta} W^{2,2+\veps}$, and then use a Sobolev embedding. For some $\delta,\veps>0$ to be determined below, we invoke Lemma ~\ref{lem:maximal regularity for Stokes} to obtain
		\begin{equation} \label{eq:maximal regularity for Galerkin approx}
			\Vert u \Vert_{L^{1+\delta}W^{2,2+\veps}}
			\lesssim
			\Vert{u_0}\Vert_{D_{2+\veps}^{\frac{\delta}{1+\delta},1+\delta}} +
			\Vert{u\cdot\nabla u}\Vert_{L^{1+\delta} L^{2+\veps}} + \Vert{\rho}\Vert_{L^{1+\delta} L^{2+\veps}} .
		\end{equation}
		We cite from ~\cite[Remark 2.5]{GS} the useful embedding (for $0<r<\infty$ and $1<q<\infty$)
		\begin{equation} \label{eq:embedding of interpolation norms}
			\Vert u_0\Vert_{D_{q}^{1-\frac{1}{r},r}} 
			\lesssim 
			\Vert A_q^{1-\frac{1}{r}+\delta_1} u_0 \Vert_{L^q} ,
		\end{equation}
		for any $\delta_1 > 0$. Given $s>0$, it is possible to choose $\delta,\veps$ small enough so that~\eqref{eq:numerology for Stokes Sobolev embedding} is satisfied, i.e.,
		\begin{equation} \label{eq:delta-varepsilon relation}
			\frac{\delta}{1+\delta} - \frac{1}{2+\veps}
			<
			\frac{s}{2} - \frac{1}{2} .
		\end{equation}
		With such a choice, we may now apply Lemma~\ref{lem:Stokes embedding} to rewrite the RHS of ~\eqref{eq:embedding of interpolation norms} in terms of $\Vert A_2^{\frac{s}{2}}u_0 \Vert_{L^2}$. Indeed, we obtain by~\cite[Theorem 3.1]{FHR}
		\begin{equation} \label{eq:interpolated data estimate}
			\Vert u_0\Vert_{D_{2+\veps}^{\frac{\delta}{1+\delta},1+\delta}} 
			\lesssim 
			\Vert A_{2+\veps}^{\frac{\delta}{1+\delta} +\delta_1}u_0\Vert_{L^{2+\veps}} 
			\lesssim 
			\Vert A_2^\frac{s}{2}u_0\Vert_{L^2} 
			\sim 
			\Vert u_0\Vert_{H^s} ,
		\end{equation}
		where $\sim$ stands for norm equivalence. We proceed on to the second term on the RHS of \eqref{eq:maximal regularity for Galerkin approx}. For $\alpha := 1-\frac{2\delta}{(1+\delta)s} \in (0,1)$, ensured by an appropriate choice of $\delta$, we have
		\begin{equation} \label{eq:maximal regularity RHS 1}
			\begin{aligned}
				\Vert u\cdot \nabla u\Vert_{L^{1+\delta} L^{2+\veps}} 
				&\lesssim 
				\Vert u\Vert_{L^{\frac{2(1+\delta)}{1-\delta}} L^{\infty}} \Vert \nabla u\Vert_{L^2 L^{2+\veps}} \\
				&\lesssim 
				\Vert u\Vert_{L^{\frac{2(1+\delta)}{1-\delta}} H^{1+\alpha s}} \Vert u\Vert_{L^2 H^{1+s}} \\
				&\lesssim 
				\Vert u\Vert_{L^{\infty} H^s}^{1-\theta} \Vert u\Vert_{L^{\frac{2(1+\delta)}{1-\delta}\theta} H^{1+s}}^{\theta} \Vert u\Vert_{L^2 H^{1+s}} \le C(T,\lVert u_0\rVert_{H^s},\lVert \rho_0\rVert_{L^2}) ,
			\end{aligned}
		\end{equation}
		where $\theta = 1-(1-\alpha)s\in (0,1)$, which implies that $\frac{2(1+\delta)}{1-\delta}\theta = 2$. The first step in~\eqref{eq:maximal regularity RHS 1} employs H\"older's inequality, while the second step follows from Sobolev embedding (for a sufficiently small $\veps$). The penultimate step interpolates the $H^{1+\alpha s}$ norm between $H^s$ and $H^{1+s}$, and the last inequality is a result of the fractional energy estimate in~\eqref{eq:H^s energy bound}. Finally, for the term involving density, we note that $\rho$ is purely transported and its $L^p$ norms are conserved. Using~\eqref{eq:density L^p norm}, we get
		\begin{equation} \label{eq:maximal regularity RHS 2}
			\Vert\rho\Vert_{L^{1+\delta} L^{2+\veps}} 
			\le 
			T^{\frac{1}{1+\delta}}\Vert \rho_0\Vert_{L^{2+\veps}} .
		\end{equation}
		For any $p>2$ and $s>0$, there exists $\veps$ small enough that $L^p,H^s\subset L^{2+\veps}$. Thus, if we begin with data $\rho_0$ that belongs to $L^p$ or $H^s$, then the RHS of~\eqref{eq:maximal regularity RHS 2} may be bounded by the corresponding norm. Therefore, from \eqref{eq:interpolated data estimate}--\eqref{eq:maximal regularity RHS 2}, and denoting $Z:= L^p$ or $H^s$ (for $p>2$ or $s>0$), we conclude
		\begin{equation} \label{eq:u^m in L^1+delta,L^2+veps}
			\Vert u \Vert_{L^{1+\delta}W^{2,2+\veps}} 
			\le
			C\left( T, \lVert u_0\rVert_{H^s}, \lVert \rho_0\rVert_{Z} \right) .
		\end{equation}
		Consequently, if $u_0\in H^s$ for $s>0$ and $\delta,\veps$ are selected to obey~\eqref{eq:delta-varepsilon relation}, then we have $u \in L^{1+\delta} W^{2,2+\veps}$. Lastly, from Gagliardo-Nirenberg (GN) interpolation,~\eqref{eq:u^m in L^1+delta,L^2+veps}, and the standard energy estimate~\eqref{eq:energy bound}, we observe that
		\begin{equation} \label{eq:existence of characteristics}
			\Vert \nabla u\Vert_{L^\infty} 
			\lesssim 
			\Vert \nabla u\Vert_{L^2}^\frac{\veps}{2+2\veps} \Vert D^2 u\Vert_{L^{2+\veps}}^\frac{2+\veps}{2+2\veps} + \Vert \nabla u\Vert_{L^2}
			\in 
			L^1 (0,T) ,
		\end{equation}
		guaranteeing the existence of characteristics of the flow.
	\end{proof}

    \section{$V\times H^1$ persistence --- Proof of Theorem ~\ref{T01}$(i)$} \label{sec:proof of V x H^1 persistence}

    \subsection{Approximate equations and a priori estimates} \label{sec:approx equations, a priori estimates}

    We begin by setting up our approximations of the system in ~\eqref{eq:boussinesq equations}. To this end, let $\{w_j\}_{j=1}^\infty$ be a 
	basis (orthonormal in $H$ and orthogonal in $V$)
	consisting of (smooth) Stokes eigenfunctions
	corresponding to the eigenvalues $\{\lambda_j\}_{j=1}^\infty$, 
	where $0<\lambda_1\leq \lambda_2\leq \cdots$. For $m \in \mathbb{N}$,
	denote by $P_m \colon L^2 \rightarrow H$ the orthogonal projection 
	onto the subspace of $H$ spanned by 
	$\{w_1, \ldots, w_m\}$. For functions $\xi_j \colon [0,T]\mapsto \mathbb{R}, j=1,2,\hdots,m$, we define $u^m(t) := \sum_{j=1}^m \xi^m_j(t) w_j$ to be our finite dimensional velocity field. Similarly, we consider the orthogonal basis of $L^2$ given by the eigenfunctions, $\{\bar{w}_j\}_{j=1}^\infty$, of the Laplace operator, and define $\rho^m(t) := \sum_{j=1}^m \bar{\xi}^m_j(t) \bar{w}_j$. Then, the Galerkin system is given by
	\begin{equation} \label{eq:Galerkin approximation}
		\begin{aligned}
			u_t^m + Au^m + P_m(u^m\cdot\nabla u^m) 
			&= 
			P_m(\rho^{m} e_2) , \\
			\rho_t^m + u^m\cdot\nabla\rho^m 
			&= 
			0 , \\
			\nabla\cdot u^m 
			&= 
			0 , \\
			(u^m,\rho^m)(0) 
			&= 
			(P_m u_0,\rho_0^m) .
		\end{aligned}
	\end{equation}
	Here, $\rho_0^m$ is a sequence of smooth approximations to $\rho_0$, which converge to $\rho_0$ in $H^1$ (and also in all $L^p$ norms). This can be achieved via mollification. Therefore, \eqref{eq:Galerkin approximation} is a system of ODE's with a unique and smooth solution on $[0,T]$, for some $T>0$. We now derive various a priori estimates (with uniform bounds) for the approximate system.
	
	\subsubsection{Density in $L^p$}
	For any $1\le p<\infty$, we multiply ~\eqref{eq:Galerkin approximation}$_{2}$ with $(\rho^m)^{p-1}$, integrate over $\Omega$ and use incompressibility to obtain
	\begin{equation} \label{eq:density L^p norm}
		\Vert\rho^m(t)\Vert_{L^p} 
		= 
		\Vert{\rho_0^m}\Vert_{L^p} \le \Vert{\rho_0}\Vert_{L^p}.
	\end{equation}
	In fact, since the $u^m$ are smooth, the flow has well-defined characteristics, and the above estimate also holds for $p=\infty$. This property is preserved in the limit $m\rightarrow\infty$, since we show using higher-order a priori bounds that $u\in L^1 W^{1,\infty}$.
	
	\subsubsection{Energy estimates}
	Next, we test ~\eqref{eq:Galerkin approximation}$_{1}$ with $u^m$, and use the continuity equation to get
	\begin{equation*}
		\begin{aligned}
			\frac{1}{2}\frac{d}{dt}\Vert{u^m}\Vert_{L^2}^2 + \Vert{\nabla u^m}\Vert_{L^2}^2 
			&= 
			\int_{\Omega} \rho^m u^m\cdot e_2 \\
			&= 
			\int_{\Omega} \rho^m u^m\cdot \nabla x_2 
			= 
			-\int_{\Omega} \nabla\cdot(\rho^m u^m) x_2 = \frac{d}{dt}\int_{\Omega} \rho^m x_2.
		\end{aligned}
	\end{equation*}
	Integrating in time yields
	\begin{equation} \label{eq:integrating energy equation}
		\begin{aligned}
			\frac{1}{2}\Vert{u^m(t)}\Vert_{L^2}^2 + \int_0^t \Vert{\nabla u^m}\Vert_{L^2}^2 
			&= 
			\frac{1}{2}\Vert{u_0^m}\Vert_{L^2}^2 + \int_{\Omega} \rho^m(t) x_2 - \int_{\Omega}\rho_0^m x_2 \\
			&\le 
			\frac{1}{2}\Vert{u_0^m}\Vert_{L^2}^2 + \sup_{\Omega}\abs{x_2} \left(\int_{\Omega}\rho^m(t) + \int_{\Omega}\rho_0^m \right) \\
			&\le 
			\frac{1}{2}\Vert{u_0}\Vert_{L^2}^2 + 2\sup_{\Omega}\abs{x_2} \int_{\Omega}\rho_0.
		\end{aligned}
	\end{equation}
	The last inequality is a consequence of $u_0^m$ being a projection of $u_0$, the domain $\Omega$ being bounded, the conservation of all $L^p$ norms of $\rho^m$, and the convergence of $\rho_0^m$ to $\rho_0$. We conclude that 
	\begin{equation}\label{eq:energy bound}
		\sup_{[0,\infty]}\Vert{u^m}\Vert_{L^2}^2 + \int_0^\infty \Vert{\nabla u^m}\Vert_{L^2}^2 
		\le 
		C,
	\end{equation}
	where $C$ depends on $\norm{u_0}_{L^2}$ and $\Vert{\rho_0}\Vert_{L^1}$.

    \subsubsection{Gradient of the density} \label{sec:gradient of the density}
	Applying the gradient to ~\eqref{eq:Galerkin approximation}$_{2}$, and testing with $\nabla\rho^m$, we have
	\begin{equation} \label{eq:gradient rho^m L^2 first step}
		\frac{1}{2}\frac{d}{dt}\Vert{\nabla\rho^m}\Vert_{L^2}^2 
		= 
		-\int_{\Omega}\partial_j \rho^m \partial_j u^m_i \partial_i \rho^m \le \Vert{\nabla u^m}\Vert_{L^{\infty}}\Vert{\nabla\rho^m}\Vert_{L^2}^2.
	\end{equation}
	Gr\"onwall's inequality and the convergence of $\rho_0^m$ to $\rho_0$ in $H^1$ yield
	\begin{equation} \label{eq:rho in H^1}
		\Vert{\nabla\rho^m}\Vert_{L^{\infty} L^2} 
		\le 
		\Vert{\nabla\rho_0}\Vert_{L^2} \exp \left( C\int_0^T \Vert{\nabla u^m}\Vert_{L^{\infty}}\right).
	\end{equation}
	From the continuity equation, we also obtain
	\begin{equation} \label{eq:rho_t in L^2 L^2}
		\Vert{\rho^m_t}\Vert_{L^2 L^2}
		\le 
		\Vert{u^m}\Vert_{L^2 L^{\infty}} \Vert{\nabla\rho^m}\Vert_{L^{\infty} L^2}
		\le 
		\Vert{\nabla\rho_0}\Vert_{L^2} \Vert{u^m}\Vert_{L^2 L^{\infty}} \exp \left(C\int_0^T \Vert{\nabla u^m}\Vert_{L^{\infty}}\right) 
	\end{equation}
	on $[0,T]$. Furthermore, we have $\Vert{u^m}\Vert_{L^2 L^{\infty}} \lesssim \Vert{u^m}\Vert_{L^2 H^2} \lesssim \Vert{Au^m}\Vert_{L^2 L^2}$, the boundedness of which is shown in ~\eqref{eq:higher order energy bounds} below. Therefore, the right-hand side of ~\eqref{eq:rho in H^1} and ~\eqref{eq:rho_t in L^2 L^2} only depends on $T$ and the initial data, according to~\eqref{eq:existence of characteristics} and \eqref{eq:integrating energy equation}.
	
	\subsubsection{Higher order energy estimates} \label{sec:higher order energy estimates}
	To bound the velocity gradient, we test ~\eqref{eq:Galerkin approximation}$_{1}$ with $Au^m$, and use the fact that $(Au^m,u^m) = \Vert{\nabla u^m}\Vert_{L^2}^2$. Therefore,
	\begin{equation*}
		\begin{aligned}
			\frac{1}{2}\frac{d}{dt}\Vert{\nabla u^m}\Vert_{L^2}^2 + \Vert{Au^m}\Vert_{L^2}^2 
			&= 
			-\left(Au^m,u^m\cdot\nabla u^m\right) + \left(Au^m,\rho^m e_2\right) \\
			&\le 
			\Vert{u^m}\Vert_{L^4}\Vert{\nabla u^m}\Vert_{L^4}\Vert{Au^m}\Vert_{L^2} + \Vert{\rho^m}\Vert_{L^2}\Vert{Au^m}\Vert_{L^2} \\
			&\le 
			C\Vert{u^m}\Vert_{L^2}^2 \Vert{\nabla u^m}\Vert_{L^2}^4 + C\Vert{\rho^m}\Vert_{L^2}^2 + \frac{1}{2}\Vert{Au^m}\Vert_{L^2}^2 \\
			&\le 
			C\left(\Vert{\nabla u^m}\Vert_{L^2}^4 + 1\right) + \frac{1}{4}\Vert{Au^m}\Vert_{L^2}^2.
		\end{aligned}
	\end{equation*}
	The penultimate step is a consequence of the Ladyzhenskaya and Young's inequalities, and the constant $C$ in the last step depends on the initial data, using the bounds in ~\eqref{eq:density L^p norm} and ~\eqref{eq:energy bound}. We also test ~\eqref{eq:Galerkin approximation}$_{1}$ with $u^m_t$, which after similar manipulations leads to 
	\begin{equation*}
		\Vert{u^m_t}\Vert_{L^2}^2 + \frac{1}{2}\frac{d}{dt}\Vert{\nabla u^m}\Vert_{L^2}^2 
		\le 
		C\left(\Vert{\nabla u^m}\Vert_{L^2}^4 + 1\right) + \frac{1}{2}\Vert{u^m_t}\Vert_{L^2}^2 + \frac{1}{4}\Vert{Au^m}\Vert_{L^2}^2.
	\end{equation*}
	Finally, we add the last two inequalities and apply Gr\"onwall's inequality to arrive at
	\begin{equation} \label{eq:higher order energy bounds}
		\sup_{0\le t\le T}\Vert{\nabla u^m}\Vert_{L^2}^2 + \int_0^T \left(\Vert{Au^m}\Vert_{L^2}^2 + \Vert{u^m_t}\Vert_{L^2}^2\right) 
		\le 
		C,
	\end{equation}
	where the constant $C$ depends on $T$, $\Vert{u_0}\Vert_{H^1}$, and $\Vert{\rho_0}\Vert_{L^2}$.

  	\subsection{Passing to the limit} \label{sec:passing to the limit}
    Recalling that \eqref{eq:Galerkin approximation} is an ODE system, \eqref{eq:density L^p norm}, \eqref{eq:energy bound}, \eqref{eq:rho in H^1}, and \eqref{eq:higher order energy bounds} imply that the time of existence is independent of $m$, and is equal to $\infty$. 
	Therefore, using these bounds, we may pass to the limit as $m\rightarrow\infty$. Indeed, for any $T>0$ fixed, there exists a pair $(u,\rho)$ satisfying 
	\begin{equation} \label{eq:weak limits}
		\begin{aligned}
			&u^m \rightharpoonup u \text{ weakly-* in } L^{\infty} V \\
			&u^m \rightharpoonup u \text{ weakly in } L^2 D(A) \cap L^{1+\delta} W^{2,2+\veps} \cap H^1 H , \\
			&u^m \to u \text{ strongly in }   L^2 V , \\
			&\rho^m \rightharpoonup \rho \text{ weakly-* in } L^\infty H^1, \\
			&\rho^m \rightharpoonup \rho \text{ weakly in } H^1 L^2 , \\
			&\rho^m \to \rho \text{ strongly in } L^\infty L^2 ,
		\end{aligned}
	\end{equation}
	on the time interval~$[0,T]$. Here, the strong convergence assertions follow from the Aubin-Lions compactness lemma. Using an analogue of the Lions-Magenes lemma for Leray-projected Hilbert spaces (see ~\cite[Lemma 6.7]{RRS}), we also conclude that $u\in C([0,T_0];V)$. Henceforth, we only consider such a continuous representative of the solution. Using the convergence specified in \eqref{eq:weak limits}, we may pass to the limit in the system concluding that $(u,\rho)$ is a solution for~\eqref{eq:boussinesq equations} in the sense of distributions.
	
	\subsection{Uniqueness of solutions} \label{sec:uniqueness of solutions}
	For $(u,\rho) \in CH \times CH^{-1}$, uniqueness was established in~\cite[Theorem 5.1]{H}. This result applies to the scenarios considered here, since we work in subspaces of $CH \times CH^{-1}$.

	\subsection{Time continuity of density} \label{sec:time continuity of density}
	In this section, we establish the continuity of the density in $C H^1$. Subtracting ~\eqref{eq:Galerkin approximation}$_{2}$ and ~\eqref{eq:boussinesq equations}, and denoting $R^m := \rho^m-\rho$ and $U^m := u^m-u$, we get
	\begin{equation} \label{eq:rho^m-rho}
		\begin{aligned}
			R^m_t 
			&= 
			- u^m\cdot\nabla R^m - U^m\cdot\nabla \rho , \\
			R^m(0) 
			&= 
			\rho_0^m - \rho_0 .
		\end{aligned}
	\end{equation}
	Now, testing \eqref{eq:rho^m-rho} with $R^m$ and using Hölder's inequality, we obtain
	\begin{equation*}
		\frac{1}{2}\frac{d}{dt}\Vert R^m\Vert_{L^2}^2 
		\le
		\Vert U^m\Vert_{L^{\infty}} \Vert\nabla\rho\Vert_{L^2} \Vert R^m\Vert_{L^2} .
	\end{equation*}
	Then, applying Gr\"onwall's inequality gives
	\begin{equation} \label{eq:R^m bound in L^2}
		\Vert R^m\Vert_{L^2} 
		\lesssim 
		\Vert R_0^m \Vert_{L^2} \exp{C\left(\Vert U^m\Vert_{L^1 L^{\infty}} \Vert\nabla\rho\Vert_{L^{\infty} L^2}\right)} .
	\end{equation}
	Thanks to ~\eqref{eq:rho in H^1} and the argument following ~\eqref{eq:rho_t in L^2 L^2}, the exponential term in the above inequality is bounded. Hence, using the convergence of $\rho_0^m \rightarrow \rho_0$ in $L^2$,	we conclude that $\rho^m \rightarrow \rho$ in $C([0,T];L^2)$. This, along with $\rho\in L^{\infty}H^1$, leads to $\rho\in C_w([0,T];H^1)$.
	
	Integrating~\eqref{eq:gradient rho^m L^2 first step}, we get
	\begin{equation} \label{eq:gradient rho^m in L^2 as function of time}
		\Vert \nabla\rho^m(t) \Vert_{L^2}
		\le 
		\Vert \nabla\rho_0^m \Vert_{L^2} \exp \left(C\int_{0}^{t} \Vert \nabla u^m(s)\Vert_{L^{\infty}} \ ds \right) .
	\end{equation}
	Passing to the limit $m\rightarrow\infty$, and using the weak-* convergence of $\rho^m$ to $\rho$ from ~\eqref{eq:weak limits}$_{4}$ and the strong convergence of $\rho_0^m$ to $\rho_0$ in $H^1$,
	\begin{equation} \label{eq:gradient rho in L^2 at any t}
		\begin{aligned}
			\Vert \nabla\rho(t) \Vert_{L^2}
			\le 
			\liminf_{m\rightarrow\infty}\Vert \nabla\rho^m(t) \Vert_{L^2}
			&\le 
			\limsup_{m\rightarrow\infty}\Vert \nabla\rho^m(t) \Vert_{L^2} \\
			&\le
			\Vert \nabla\rho_0 \Vert_{L^2} \limsup_{m\rightarrow\infty} \exp \left(C \int_{0}^{t} \Vert \nabla u^m(s)\Vert_{L^{\infty}} \ ds \right) .
		\end{aligned}
	\end{equation}
	The exponential term is bounded (uniformly in $m$) on the interval $[0,T]$. Thus,    
	\begin{equation} \label{eq:gradient rho in L^2 at t=0}
		\limsup_{t\rightarrow 0^+}\Vert \nabla\rho(t) \Vert_{L^2}
		\le 
		\Vert \nabla\rho_0 \Vert_{L^2} .
	\end{equation}
	Since $\rho\in C_w H^1$, lower semicontinuity provides the bound $\Vert \nabla\rho_0 \Vert_{L^2} \le \liminf_{t\rightarrow 0^+} \Vert\nabla\rho(t) \Vert_{L^2}$. Combining this with ~\eqref{eq:gradient rho in L^2 at t=0} yields norm continuity at $t=0^+$, and thus strong right continuity in time at $t=0$ of the $H^1$ norm of the density. With the help of the time-reversal symmetry, we also obtain left continuity at $t=0$. Next, we exploit the time-translation symmetry of the system. Evolving the system from $\rho_0\in H^1$ for some time $\tau>0$, and choosing the continuous representative in $L^2$, we know that the solution $\rho(\tau)\in H^1$ in accordance with the a priori estimates in Section~\ref{sec:approx equations, a priori estimates}. With $\rho(\tau)$ as a new initial condition, we repeat the above arguments to obtain strong continuity at $t=\tau$. Moreover, due to the uniqueness of the solution (Section~\ref{sec:uniqueness of solutions}), it is clear that the evolutions of $\rho_0$ and $\rho(\tau)$ are identical for $t\ge\tau$.
	Since $\tau>0$ was arbitrary, we conclude that $\rho\in C([0,T];H^1)$.

	\subsection{Improved velocity regularity and asymptotic properties} \label{sec:improved velocity regularity for t>0}
	The solution obtained thus far is not only global-in-time but enjoys the same asymptotic properties stated in Theorem ~\ref{T00}. Indeed, since $u\in L^2 D(A)$, for each $\epsilon \in (0,1)$, we can find a $0<t_\epsilon<\epsilon$ such that $(u(t_\epsilon),\rho(t_\epsilon))$ belongs to $D(A) \times H^1$. Now, we consider the system
	\begin{equation} \label{eq:epsilon-boussinesq system}
		\begin{aligned}
			u^\epsilon_t + A u^\epsilon + \Leray(u^\epsilon \cdot \nabla u^\epsilon) 
			&= 
			\Leray(\rho^{\epsilon} e_2) , \\
			\rho^\epsilon_t + u^\epsilon \cdot \nabla \rho^\epsilon 
			&= 
			0 , \\
			(u^\epsilon,\rho^\epsilon)(0) 
			&= 
			(u(t_\epsilon), \rho(t_\epsilon)) ,
		\end{aligned}
	\end{equation}
	which has a unique solution in the class $V\times H^1$ by the preceding arguments. In addition, since 
	\begin{equation} \label{eq:u,rho in CH, CL^2}
		(u,\rho) \in C([0,T];V)\times C([0,T];H^1) ,	
	\end{equation}
	we have
	\begin{equation} \label{eq:continuity of epsilon-solution and original solution}
		\lim_{t \to t_\epsilon} (u(t),\rho(t)) 
		= 
		(u(t_\epsilon), \rho(t_\epsilon)) 
		= 
		(u^\epsilon,\rho^\epsilon)(0) .	
	\end{equation}
	This shows that $(u, \rho)$ is the unique solution of ~\eqref{eq:epsilon-boussinesq system} for $(x,t) \in \Omega\times [t_\epsilon,T]$. Therefore, we may apply Theorem ~\ref{T00} and deduce that the solution is global-in-time satisfying further asymptotic properties.
	
	\subsection{Time derivatives} \label{sec:time derivatives in V' and H^-1}
	We rewrite~\eqref{eq:boussinesq equations} as
	\begin{equation} \label{eq:equations for u_t and rho_t}
		\begin{aligned}
			u_t 
			&= 
			-A u -\Leray(u \cdot \nabla u) + \Leray(\rho e_2) , \\
			\rho_t
			&= 
			-\nabla\cdot(u\rho) ,
		\end{aligned}
	\end{equation}
	and use the regularity of solutions. Indeed, for the density,
    \begin{equation*}
        \Vert \rho_t (t) \Vert_{H^{-1}} \lesssim \Vert u(t)\rho(t) \Vert_{L^2} \lesssim \Vert u(t) \Vert_{L^4} \Vert \rho(t) \Vert_{L^4} \lesssim \Vert u(t) \Vert_{H^1} \Vert \rho(t) \Vert_{H^1} .
    \end{equation*}
    Since the RHS is strongly continuous in time, so is the LHS. The proof for the velocity is similar.
    \qed

    \section{$D(A)\times H^2$ persistence --- Proof of Theorem ~\ref{T01}$(ii)$} \label{sec:proof of D(A) x H^2 persistence}
	
	In this section, we begin from initial data $(u_0,\rho_0)\in D(A)\times H^2$. Since this is more regular than the data in Theorem ~\ref{T00}, we already know that $(u,\rho)$ belongs to $L^{\infty} D(A) \times C H^1$. Therefore, it suffices to show that $\rho$ belongs to $C H^2$. Indeed, the strong continuity of $u$ follows from \cite[Theorem 3.1]{T3}. We work with the approximate system in ~\eqref{eq:Galerkin approximation} to derive a priori estimates. The details of the construction are identical to that in Section~\ref{sec:proof of V x H^1 persistence}, hence they are not repeated here. We remind the reader that all implicit constants depend on $T$ and the initial data (and of course, on the domain $\Omega$).
	
	We differentiate ~\eqref{eq:Galerkin approximation}$_2$ twice to get
	\begin{equation} \label{eq:second derivative of density equation}
		\partial_j \partial_k \rho^m_t + u^m\cdot\nabla \partial_j \partial_k \rho^m 
		= 
		-(\partial_j u^m)\cdot\nabla \partial_k\rho^m - (\partial_k u^m)\cdot\nabla \partial_j\rho^m - (\\
		\partial_j \partial_k u^m)\cdot\nabla\rho^m ,
	\end{equation}
	for $j,k \in \{1,2\}$. Next, testing ~\eqref{eq:second derivative of density equation} with $\partial_j\partial_k\rho^m$, and summing over $j,k$ yields
	\begin{equation} \label{eq:D^2 rho step 1}
		\begin{aligned}
			\frac{1}{2}\frac{d}{dt}\Vert D^2 \rho^m\Vert_{L^2}^2
			&= 
			-(\partial_k\partial_j u^m\cdot\nabla\rho^m,\partial_k\partial_j\rho^m)
			-(\partial_j u^m\cdot\nabla\partial_k\rho^m,\partial_k\partial_j\rho^m) \\
			&\quad
			-(\partial_k u^m\cdot\nabla\partial_j\rho^m,\partial_k\partial_j\rho^m) .
		\end{aligned}
	\end{equation}
	Note that the term involving three derivatives of $\rho^m$ vanishes due to incompressibility. Thus we may rewrite ~\eqref{eq:D^2 rho step 1} as
	\begin{equation} \label{eq:introduction of I in 2nd order density estimate}
		\frac{1}{2}\frac{d}{dt}\Vert D^2\rho^m\Vert_{L^2}^2 
		= 
		-(\partial_k\partial_j u^m\cdot\nabla\rho^m,\partial_k\partial_j\rho^m) + I ,
	\end{equation}
	from where it is easy to see that 
	$\abs{I} \lesssim	\Vert \nabla u^m\Vert_{L^\infty} \Vert D^2 \rho^m\Vert_{L^2}^2$.
	For the remaining term, we have after using the Ladyzhenskaya inequality,
	\begin{equation} \label{eq:D^2 rho step 2}
		\begin{aligned}
			(\partial_k\partial_j u^m\cdot\nabla\rho^m,\partial_k\partial_j\rho^m) &\lesssim 
			\Vert D^2 u^m\Vert_{L^{2+\veps}} \Vert \nabla \rho^m\Vert_{L^{2+\frac{4}{\veps}}} \Vert D^2\rho^m\Vert_{L^2} \\
			&\lesssim
			\Vert u^m\Vert_{W^{2,2+\veps}} \Vert \nabla \rho^m\Vert_{H^1} \Vert D^2\rho^m\Vert_{L^2} \\
			&\lesssim
			\Vert u^m\Vert_{W^{2,2+\veps}} \Vert D^2 \rho^m\Vert_{L^2}^2 + \Vert u^m\Vert_{W^{2,2+\veps}} \Vert \nabla \rho^m\Vert_{L^{\infty} L^2} \Vert D^2 \rho^m\Vert_{L^2} .
		\end{aligned}
	\end{equation}
	We thus arrive at a linear ODE-type inequality
	\begin{equation} \label{eq:D^2 rho step 3}
		\frac{d}{dt}\Vert D^2\rho^m\Vert_{L^2} 
		\lesssim 
		a_m \Vert D^2\rho^m\Vert_{L^2} + b_m ,
	\end{equation}
	where $a_m = \Vert u^m\Vert_{W^{1,\infty}} + \Vert u^m\Vert_{W^{2,2+\veps}}$, and $b_m=  \Vert \nabla \rho^m\Vert_{L^{\infty} L^2} \Vert u^m\Vert_{W^{2,2+\veps}}$. We point out that $a_m, b_m \in L^1(0,T)$ uniformly in $m$, according to~\eqref{eq:u^m in L^1+delta,L^2+veps} and~\eqref{eq:existence of characteristics}. Hence, by Gronwall's inequality, we obtain
	\begin{equation} \label{eq:rho in L^infty H^2}	
		\Vert D^2\rho^m(t)\Vert_{L^2}
		\le
		\left(C\int_{0}^t b_m(s)\,ds + \Vert D^2\rho_0^m\Vert_{L^2}\right) \exp{\left(C\int_0^t a_m(s)\,ds\right)}
		\lesssim
		1 ,
	\end{equation}
	where the constants depend on $T$ and the initial data. It follows that the weak-* limit $\rho$ belongs to $L^{\infty} H^2$. 
	
	It remains to establish the continuity of $\rho$. The arguments in Section~\ref{sec:time continuity of density} showed that $\rho\in CH^1$. Combining with~\eqref{eq:rho in L^infty H^2}, we conclude that $\rho\in C_w([0,T];H^2)$, which gives us $\Vert D^2\rho_0 \Vert_{L^2}\le\liminf_{t\rightarrow 0^+} \Vert D^2\rho(t) \Vert_{L^2}$. Next, we use similar arguments as in Section~\ref{sec:time continuity of density}. From ~\eqref{eq:rho in L^infty H^2}, we see that
	\begin{equation} \label{eq:D^2 rho in L^2 at any t}
		\begin{aligned}
			\Vert D^2\rho(t) \Vert_{L^2}
			&\le 
			\liminf_{m\rightarrow\infty}\Vert D^2\rho^m(t) \Vert_{L^2}
			\le 
			\limsup_{m\rightarrow\infty}\Vert D^2\rho^m(t) \Vert_{L^2} \\
			&\le
			\left(C\limsup_{m\rightarrow\infty}\int_{0}^t b_m(s)\,ds + \Vert D^2\rho_0\Vert_{L^2}\right) \exp{\left(C\limsup_{m\rightarrow\infty}\int_0^t a_m(s)\,ds\right)} .
		\end{aligned}
	\end{equation}
	Again, both integrals in the last step are bounded uniformly in $m$. Therefore, passing the limit as $t\rightarrow 0^+$, we obtain
	\begin{equation} \label{eq:D^2 rho in L^2 at t=0}
		\limsup_{t\rightarrow 0^+}\Vert D^2\rho(t) \Vert_{L^2}
		\le 
		\Vert D^2\rho_0 \Vert_{L^2} .
	\end{equation}
	That $\rho\in C([0,T];H^2)$ follows from time-reversal and time-translation symmetries, and the uniqueness of the solution, just as in Section~\ref{sec:time continuity of density}. Finally, we also remark that the same arguments as in Section~\ref{sec:time derivatives in V' and H^-1} allow us to conclude that $u_t\in C H$ and $\rho_t \in C L^2$.
	\qed

\section{$H^k\times H^k$ persistence --- Proof of Theorem~\ref{T01}$(iii)$} \label{sec:proof of H^k x H^k persistence}
	
	In this section, we show that if the initial data satisfies
	$(u_0,\rho_0) \in (H^k \cap D(A)) \times H^k$, for a fixed $k \in \{3,4,\ldots\}$,
	and the $k^{th}$-order compatibility conditions
	\begin{equation} \label{eq:compatibility conditions for data}
		\begin{aligned}
			\partial^j_t u(0) &\in V \indent \text{ for } 1 \le \abs{j} \le {\left\lfloor \frac{k}{2}\right\rfloor}-1
			\\
			\partial^\frac{k}{2}_t u(0) &\in H \indent \text{ if $k$ is even }
			\\
			\partial^\frac{k-1}{2}_t u(0) &\in V \indent \text{ if $k$ is odd } ,
		\end{aligned}
	\end{equation} 
	hold, then the solution $(u,\rho)$ for \eqref{eq:boussinesq equations} belongs to $C([0,T];H^k\cap D(A)) \times C([0,T];H^k)$.
	In fact, we prove a stronger result. 
	Before stating the exact theorem, we introduce some notation.
	Following \cite{T3}, for $T>0$ fixed, we define 
	\begin{equation} \label{eq:defining W_k and Tilde W_k}
		W_k 
		=  
		\bigcap_{j=0}^{\lfloor\frac{k}{2}\rfloor} C^j([0,T]; H^{k-2j} \cap H), \quad \tilde{W}_k 
		= 
		\bigcap_{j=0}^{\lfloor\frac{k}{2}\rfloor} C^j([0,T]; H^{k-2j}) .
	\end{equation}
	We use an induction argument to establish the a priori estimates. To prove $u\in W_k$, we cite~\cite{T3}, adapted to our setting.
    
	\begin{Theorem}[A characterization of the compatibility conditions] \label{T04}
		Consider
		\begin{equation}
			\begin{aligned} \label{eq:momentum equation alone}
				u_t + Au + \mathbb{P}(u \cdot \nabla u) 
				&= 
				\mathbb{P}(\rho e_2)
				\\
				u(0)
				&=
				u_0.
			\end{aligned}
		\end{equation}
		For some $k\ge 2$, assume that 
		$u_0 \in H^k \cap D(A)$, $\mathbb{P}(\rho e_2) \in W^{k-2}$ and
		\begin{equation} \label{eq:compatibility conditions for momentum inhomogeneity}
			\begin{aligned}
				&\partial^{l}_t \mathbb{P}(\rho e_2) \in L^2 H, \indent k \text{ is odd and } l = {\left\lfloor \frac{k}{2}\right\rfloor}	
				\\
				&\partial^{l}_t \mathbb{P}(\rho e_2) \in L^2 V', \indent k \text{ is even and } l = \frac{k}{2} .
			\end{aligned}
		\end{equation}
		Then a necessary and sufficient condition for the solution of \eqref{eq:momentum equation alone}
		to belong to $W_k$ is \eqref{eq:compatibility conditions for momentum inhomogeneity}.
	\end{Theorem}
	
	In what follows, we need the fact that
	$\Leray \colon H^{-1}(\Omega)\to V'$ is continuous. This is not difficult to see; indeed, for $f\in C^{\infty}(\Omega)$ and $v\in V$,
	\begin{equation} \label{eq:continuity of Leray}
		\abs{(\Leray f,v)}
		=
		\abs{( f,\Leray v)}
		=
		\abs{( f,v)}
		\lesssim 
		\Vert f\Vert_{H^{-1}} \Vert v\Vert_{H_{0}^{1}}
		= 
		\Vert f\Vert_{H^{-1}}\Vert v\Vert_{V} .
	\end{equation}
	This shows that for any $f\in C^{\infty}(\Omega)$, we have
	\begin{equation} \label{eq:Leray f bounded in V'}
		\Vert \Leray f\Vert_{V'}
		\lesssim
		\Vert f\Vert_{H^{-1}} ,
	\end{equation}
	thus establishing the continuity of $\Leray \colon H^{-1}(\Omega)\to V'$.
	
	\begin{proof}[Proof of Theorem~\ref{T01}$(iii)$]
		When $k=3$, we use \cite[Theorem~1.1]{LPZ} which establishes that $(u,\rho) \in C([0,T];H^3)\times C([0,T];H^3)$ under the same assumptions. Then, we only need to check that $(u_t,\rho_t)$ belongs to $C([0,T];H) \times C([0,T];L^2)$. This is evident from the regularity of $u$ and $\rho$, in conjunction with ~\eqref{eq:equations for u_t and rho_t}. In effect, we have $(u,\rho) \in W_3 \times \tilde{W}_3$.
		
		Now, we proceed to the inductive step, and assume that Theorem~\ref{T01}$(iii)$ holds for some $k \in \mathbb{N}$ where $k\ge 3$. Moreover, let $(u_0,\rho_0) \in (H^{k+1}\times D(A)) \times H^{k+1}$ satisfy the compatibility conditions $\eqref{eq:compatibility conditions for data}$ for $k+1$. The rest of the proof is organized as follows. We begin by proving that $u \in W_{k+1}$. Then, we show that $\rho \in L^\infty H^{k+1}$, after which we verify the continuity of $\rho$, i.e., if $\rho \in C([0,T];H^{k+1})$. Finally, we establish that $\rho \in \tilde{W}_{k+1}$. \\
		
		\texttt{Step~1.} We prove that $u \in W_{k+1}$.\\
		
		To achieve this, we apply Theorem~\ref{T04}, which requires that
		$\Leray (\rho e_2)$ belongs to $W_{k-1}$. This follows from the continuity of the Leray projector $\Leray$ on $L^2$-based
		Sobolev spaces (as demonstrated earlier), and the observation that the spaces $W_k$ and $\tilde{W}_k$ are nested. In particular, $\rho \in \tilde{W}_k$ implies 
		$\mathbb{P}(\rho e_2) \in W_k \subseteq W_{k-1}$.
		To invoke Theorem~\ref{T04}, we need the relations
		\begin{equation} \label{eq:compatibility for RHS of momentum, inductive step}
			\begin{aligned}
				&
				\partial^{l}_t \mathbb{P}(\rho e_2) \in L^2 H, \indent k \text{ is even and } l = {\left\lfloor \frac{k+1}{2}\right\rfloor}	
				\\
				&\partial^{l}_t \mathbb{P}(\rho e_2) \in L^2 V', \indent k \text{ is odd and } l = \frac{k+1}{2} .
			\end{aligned}
		\end{equation}
		Since $\mathbb{P}$ is continuous, it is sufficient to show \eqref{eq:compatibility for RHS of momentum, inductive step} when
		$H$ and $V'$ are replaced by $L^2$ and $H^{-1}$, respectively. To establish this, we first assume that $k$ is even, so that $l = {\left\lfloor \frac{k+1}{2}\right\rfloor} = \frac{k}{2}$. Then, since $\rho \in \tilde{W}_k$, it follows that $\partial^{l}_t\rho \in C([0,T];L^2)$. On the other hand, if $k$ is odd, then with $l=\frac{k+1}{2}$, we need to show 
		\begin{equation} \label{eq:rho^l_t for odd k}
			-\partial^{l}_t \rho
			=
			\partial^{l-1}_t (u \cdot \nabla \rho)
			=
			\sum_{i=0}^{l-1} {l-1 \choose i} \partial^{l-1-i}_t u \cdot \partial^{i}_t\nabla \rho
			\in C([0,T];H^{-1}) .
		\end{equation}
		Since $(u,\rho) \in W_{k}\times \tilde{W}_k$, we have for each $0\le i\le l-1$,
		\begin{equation} \label{eq:highest order time derivatives of u,rho for odd k}
			(\partial^{l-1-i}_t u,
			\partial^{i}_t \rho)  
			\in C([0,T];H^1\cap H) \times C([0,T];H^1) .
		\end{equation}
		Using incompressibility, \eqref{eq:rho^l_t for odd k} is established. As a result, \eqref{eq:compatibility for RHS of momentum, inductive step} holds so that by Theorem~\ref{T04}, $u \in W_{k+1}$.\\
		
		\noindent\texttt{Step~2.} We show that $\rho \in C([0,T];H^{k+1})$.\\
		
		We return to \eqref{eq:Galerkin approximation}$_2$ and consider an approximating sequence
		\begin{equation} \label{eq:approximating sequence (u^m,rho^m)}
			(u^m, \rho_0^m) 
			\to
			(u,\rho_0)
			\in
			L^2H^{k+1} \times H^{k+1}
			\comma 
			m \to \infty .
		\end{equation}
		Picking a multi-index $\alpha\in\mathbb{N}^2_0$ with $\abs{\alpha}=k+1$
		and differentiating \eqref{eq:Galerkin approximation}$_2$ gives
		\begin{equation} \label{eq:multi-index derivative of approximate density}
			\begin{aligned}
				\partial^\alpha \rho^m_t
				&=
				-\sum_{\beta < \alpha, \abs{\beta}=1 } {\alpha \choose \beta}
				\partial^\beta u^m \cdot \nabla \partial^{\alpha-\beta}\rho^m
				-\sum_{\beta < \alpha, |\beta|>1} {\alpha \choose \beta}
				\partial^\beta u^m \cdot \nabla \partial^{\alpha-\beta}\rho^m
				\\
				&\quad -u^m\cdot \nabla \partial^\alpha \rho^m
				-\partial^\alpha u^m \cdot \nabla \rho^m
				= I_1 + I_2 + I_3 + I_4.
			\end{aligned}
		\end{equation}
		We test \eqref{eq:multi-index derivative of approximate density} by $\partial^\alpha \rho^m$, and write
		\begin{equation} \label{eq:D^alpha rho^m in L^2, setup}
			\frac{1}{2}\frac{d}{dt}\Vert \partial^\alpha\rho^m\Vert_{L^2}^2
			=
			-\sum_{i=1}^{4}
			(I_i, \partial^\alpha\rho^m) .
		\end{equation}
		Observe that for the first term we have
		\begin{equation} \label{eq:I1 term}
			(I_1, \partial^\alpha\rho^m)
			\lesssim
			\Vert u^m\Vert_{W^{1,\infty}}
			\Vert \rho^m\Vert_{H^{k+1}}^2
			,
		\end{equation}
		while estimating the second term gives
		\begin{equation} \label{eq:I2 term}
			(I_2, \partial^\alpha\rho^m)
			\lesssim
			\Vert u^m\Vert_{W^{k,4}}
			\Vert \rho^m\Vert_{W^{k,4}}
			\Vert \rho^m\Vert_{H^{k+1}}
			\lesssim
			\Vert u^m\Vert_{H^{k+1}}
			\Vert \rho^m\Vert_{H^{k+1}}^2 .
		\end{equation}
		The third term (corresponding to $I_3$) vanishes due to incompressibility upon integration by parts. Finally, for the remaining term we write
		\begin{equation} \label{eq:I4 term}
			(I_4, \rho^m)
			\lesssim
			\Vert u^m\Vert_{H^{k+1}}
			\Vert \rho^m\Vert_{W^{1,\infty}}
			\Vert \rho^m\Vert_{H^{k+1}} .
		\end{equation}
		Now, combining the estimates \eqref{eq:I1 term}--\eqref{eq:I4 term},
		summing~\eqref{eq:D^alpha rho^m in L^2, setup} over $\abs{\alpha} \le m+1$,
		and using the embedding 
		$H^{k+1} \subset W^{1,\infty}$,
		we arrive at
		\begin{equation} \label{eq:D^alpha rho^m in L^2, final}
			\frac{d}{dt}\Vert \rho^m\Vert_{H^{k+1}}^2
			\lesssim
			\Vert u^m\Vert_{H^{k+1}}
			\Vert \rho^m\Vert_{H^{k+1}}^2 .
		\end{equation}
		An application of the Gronwall's inequality yields
		\begin{equation} \label{eq:rho^m in L^{infty} H^{k+1}}
			\Vert \rho^m\Vert_{L^\infty H^{k+1}}^2
			\lesssim
			\exp\left(
			C\int_0^T \Vert u^m\Vert_{H^{k+1}}\,ds
			\right)
			\Vert \rho_0^m\Vert_{H^{k+1}}^2 .           
		\end{equation}
		Passing to the limit by using \eqref{eq:approximating sequence (u^m,rho^m)}
		establishes $\rho \in L^\infty H^{k+1}$. This, combined with $\rho\in \Tilde{W}_k \subset C([0,T];H^k)$, implies $\rho\in C_w([0,T];H^{k+1})$. This also gives us lower semicontinuity of the $H^{k+1}$ norm as $t\rightarrow 0^+$. Due to the uniform (in $m$) bounds of the RHS of~\eqref{eq:rho^m in L^{infty} H^{k+1}}, we follow the same arguments as in~\eqref{eq:gradient rho in L^2 at any t} to obtain upper semicontinuity, whence we conclude the strong continuity in time of the $H^{k+1}$ norm as $t\rightarrow 0^+$. Once again, time reversal and time translation (and uniqueness of the solution) mean that the continuity indeed holds for all $t\in [0,T]$. \\
			
		\noindent\texttt{Step~3.} We show that $\rho \in \tilde{W}_{k+1}$.\\
		
		Recalling the definition of $\tilde{W}_{k+1}$, we aim to show
		\begin{equation} \label{eq:partial_t^j+1 rho in C H^k+1-2j}
			\partial_t^j \rho \in C([0,T];H^{k+1-2j}) , 
		\end{equation}
		for $j=1,\ldots, {\lfloor \frac{k+1}{2}\rfloor}$. We establish this by induction on $j$. First, when $j=1$, we need $\partial_t \rho \in C([0,T];H^{k-1})$, which follows from $u \cdot \nabla \rho$ belongs to
		$C([0,T]; H^{k-1})$. The latter is a consequence of $\nabla \rho \in C([0,T];H^{k})$, $u \in C([0,T];H^{k+1}\cap H)$ and the fact that $H^{k-1}$ is an algebra. Now, we suppose that \eqref{eq:partial_t^j+1 rho in C H^k+1-2j} holds for some $1\le j < {\lfloor \frac{k+1}{2} \rfloor}$, and prove it for $j+1$. To begin with, we apply $\partial^j_t$ to \eqref{eq:boussinesq equations}$_2$ obtaining
		\begin{equation} \label{eq:j time derivatives on continuity}
			\partial^{j+1}_t \rho = - \partial^j_t (u \cdot \nabla \rho) .
		\end{equation}
		
    	Therefore, our goal is to establish $\partial^j_t (u \cdot \nabla \rho)\in C([0,T];H^{k-1-2j})$. Rewriting this term, we get
    	\begin{equation} \label{eq:D^j (u cdot nabla rho)}
    		\partial^j_t (u \cdot \nabla \rho)(t)
    		=
    		\sum_{i=0}^{j} {j \choose i} \partial^{j-i}_t u(t) \cdot \partial^{i}_t \nabla \rho(t) .
    	\end{equation}
    	For the sake of the subsequent calculations, we recall that due to the choices for $i$, $j$ and $k$, we have $k-1-2j\ge 0$ and $k-1-2i\ge 0$. We also note from $u\in W_{k+1}$ and~\eqref{eq:defining W_k and Tilde W_k} that $\partial_t^{j-i}u\in C H^{k+1-2j+2i}$ and $\partial_t^i \nabla\rho \in C H^{k-1-2i}$. Therefore, employing the Kato-Ponce inequality (justified by extending $u$ and $\rho$ to $\mathbb{R}^2$),
    	\begin{equation} \label{eq:kato-ponce applied to partial derivatives of u and rho}
    		\begin{aligned}
    			\Vert \partial^{j-i}_t u(t) \cdot \partial^{i}_t \nabla \rho(t) \Vert_{H^{k-1-2j}}
    			&\lesssim
    			\Vert \partial^{j-i}_t u(t) \Vert_{W^{k-1-2j,\infty}} \Vert \partial^{i}_t \nabla \rho(t) \Vert_{L^2} + \Vert \partial^{j-i}_t u(t) \Vert_{L^{\infty}} \Vert \partial^{i}_t \nabla \rho(t) \Vert_{H^{k-1-2j}} \\
    			&\lesssim 
    			\Vert \partial^{j-i}_t u(t) \Vert_{H^{k+1-2j+2i}} \Vert \partial^{i}_t \nabla \rho(t) \Vert_{H^{k-1-2i}} ,
    		\end{aligned}
    	\end{equation}
    	whence we see that $\partial^{j-i}_t u \cdot \partial^{i}_t \nabla \rho \in L^{\infty} H^{k-1-2j}$. To establish continuity at time $t\in [0,T]$, we consider a sequence of times $t_n$ converging to $t$, and use the continuity of $u$ and $\rho$ as mentioned above.
    	
    	Going back to \eqref{eq:j time derivatives on continuity} and \eqref{eq:D^j (u cdot nabla rho)}, it follows that~\eqref{eq:partial_t^j+1 rho in C H^k+1-2j} holds for $j+1$, and this completes our induction for all $1\le j \le {\lfloor \frac{k+1}{2} \rfloor}$. As a consequence, we obtain $\rho \in \tilde{W}_{k+1}$, which finalizes our induction on $k$.
    \end{proof}
 
\section{Fractional regularity --- Proof of Theorem~\ref{T02}} \label{sec:proof of fractional regularity}

In this section, we only present a~priori estimates for the approximate system \eqref{eq:Galerkin approximation}, and note that construction of global-in-time unique solutions can be achieved using the methods in Section~\ref{sec:proof of V x H^1 persistence}.

	\subsection{Lower fractional regularity} \label{sec:lower fractional regularity}
	Starting with the velocity estimates, we consider $u_0$ that belongs to $D(A^{\frac{s}{2}})$ for $s\in (0,1)$. Testing~\eqref{eq:Galerkin approximation}$_1$ with $A^s u^m$ yields
	\begin{equation} \label{eq:testing momentum with A^s u^m}
		\begin{aligned}
			\frac{1}{2}\frac{d}{dt}\lVert A^\frac{s}{2}u^m \rVert_{L^2}^2 + \lVert A^\frac{s+1}{2}u^m \rVert_{L^2}^2 
			&= 
			-\int_{\Omega}A^s u^m\cdot P_m(u^m\cdot\nabla u^m) + \int_{\Omega} A^s u^m\cdot P_m(\rho^m e_2) \\
			&\le
			\frac{1}{2}\lVert A^\frac{s+1}{2}u^m \rVert_{L^2}^2 + C\lVert A^{\frac{s-1}{2}} P_m(u^m\cdot\nabla u^m)\rVert_{L^2}^2 \\ 
			&\quad + 
			C\lVert A^{\frac{s-1}{2}} P_m(\rho^m e_2)\rVert_{L^2}^2 .
		\end{aligned}
	\end{equation}
	With the help of Lemma~\ref{lem:Stokes embedding} and H\"older's inequality, we have
	\begin{equation*}
		\lVert A^{\frac{s-1}{2}} P_m(u^m\cdot\nabla u^m)\rVert_{L^2} 
		\lesssim 
		\lVert u^m\cdot\nabla u^m \rVert_{L^{\frac{2}{2-s}}} 
		\lesssim
		\lVert u^m \rVert_{L^{\frac{2}{1-s}}} \lVert \nabla u^m \rVert_{L^2} 
		\lesssim 
		\lVert A^{\frac{s}{2}}u^m \rVert_{L^2} \lVert \nabla u^m \rVert_{L^2} .
	\end{equation*}
	Since $\frac{s-1}{2}>-\frac{1}{2}$, the results in~\cite{GuS} imply that $\lVert A^{\frac{s-1}{2}} f\rVert_{L^2} \lec \lVert f \rVert_{H^{s-1}}$. Thus, we can write
	\begin{equation*}
		\lVert A^{\frac{s-1}{2}} P_m(\rho^me_2)\rVert_{L^2} 
		\lesssim 
		\lVert \rho^m \rVert_{H^{s-1}} 
		\lesssim
		\lVert \rho^m \rVert_{L^2}
		\le
		\lVert \rho_0^m \rVert_{L^{2}}
		\le 
		\lVert \rho_0 \rVert_{L^{2}} .
	\end{equation*}
	Finally,~\eqref{eq:testing momentum with A^s u^m} becomes
	\begin{equation}
		\frac{d}{dt}\lVert A^\frac{s}{2}u^m \rVert_{L^2}^2 + \lVert A^\frac{s+1}{2}u^m \rVert_{L^2}^2 
		\lesssim 
		\lVert \rho_0 \rVert_{L^2}^2 + \lVert \nabla u^m \rVert_{L^2}^2 \lVert A^{\frac{s}{2}}u^m \rVert_{L^2}^2 .
	\end{equation}
	Using Gr\"onwall's inequality, along with~\eqref{eq:energy bound}, we arrive at
	\begin{equation} \label{eq:H^s energy bound}
		\Vert{u^m}\Vert_{L^{\infty}H^s} + \lVert{u^m}\rVert_{L^2 H^{s+1}} 
		\le 
		C(T,\lVert u_0\rVert_{H^s},\lVert \rho_0\rVert_{L^2}) .
	\end{equation} 
    Therefore, we may now extract a weakly convergent subsequence, which yields a solution $u\in L^{\infty} D(A^{\frac{s}{2}}) \cap L^2 D(A^{\frac{s+1}{2}})$. Using~\eqref{eq:Galerkin approximation}$_1$, we also see that $\partial_t u \in L^2 D(A^{\frac{s-1}{2}})$, whence an application of the Lions-Magenes lemma (see~\cite[Chapter 3]{T1}) allows us to conclude the strong continuity in time. If $\rho_0\in H^{\lfloor s\rfloor} =~L^2$, we may use~\eqref{eq:density L^p norm} to extract a weak-* convergent subsequence, yielding $\rho\in L^{\infty} L^2$. After establishing $\rho_t\in L^2 H^{-1}$, it is straightforward to conclude the time continuity of $\rho$. Uniqueness follows from Section~\ref{sec:uniqueness of solutions}. This completes the arguments when $(u_0,\rho_0)\in (H^s\cap H)\times H^{\lfloor s\rfloor}$.
 
    Next, assuming that $\rho_0 \in H^s$, we turn our attention to the fractional estimates for the density. The continuity equation lacks dissipation, and we require the Kato-Ponce commutator estimate to deal with the fractional derivative. We begin by considering a divergence-free Sobolev extension of the velocity (the solution obtained above) and the initial data to $\RR^2$, referring to them as $\overline{u}^m$, and $\overline{\rho_0}^m$, respectively. That the velocity extension is divergence-free is guaranteed by~\cite{KMPT}. Therefore, the continuity equation is written down for the extended density field as
	\begin{equation} \label{eq:extended approximate continuity equation}
		\begin{aligned}
			\partial_t \overline{\rho}^m + \overline{u}^m\cdot\nabla \overline{\rho}^m &= 0 , \\
			\overline{\rho}^m(0) &= \overline{\rho_0}^m .
		\end{aligned}
	\end{equation}
	We act upon~\eqref{eq:extended approximate continuity equation}$_1$ by $\Lambda^s$ (where $\Lambda = \sqrt{-\Delta}$), and then test with $\Lambda^s \overline{\rho}^m$. Due to incompressibility, we may add a vanishing term to obtain
	\begin{equation}
		\frac{1}{2}\frac{d}{dt}\Vert \Lambda^s \overline{\rho}^m \Vert_{L^2}^2 = -\int_{\RR^2} (\Lambda^s\nabla\cdot(\overline{u}^m \overline{\rho}^m) - \overline{u}^m\cdot\Lambda^s\nabla\overline{\rho}^m) \ \Lambda^s\overline{\rho}^m ,
	\end{equation}
	which can be simplified to read
	\begin{equation} \label{eq:lower fractional density energy estimate 1}
		\begin{aligned}
			\frac{d}{dt}\Vert \Lambda^s \overline{\rho}^m \Vert_{L^2} 
			&\le 
			\Vert \Lambda^s\nabla\cdot(\overline{u}^m\overline{\rho}^m) - \overline{u}^m\cdot\Lambda^s\nabla\overline{\rho}^m \Vert_{L^2} \\
			&\lesssim 
			\Vert D\overline{u}^m \Vert_{L^{\infty}} \Vert D^s \overline{\rho}^m \Vert_{L^2} + \Vert D^{1+s} \overline{u}^m \Vert_{L^{\frac{2}{s}}} \Vert \overline{\rho}^m \Vert_{L^{\frac{2}{1-s}}} \\
			&\lesssim 
			\left( \Vert D\overline{u}^m \Vert_{L^{\infty}} + \Vert \overline{u}^m \Vert_{H^2} \right) \Vert \overline{\rho}^m \Vert_{H^s} ,
		\end{aligned}
	\end{equation}
    where the norms are understood to be over $\mathbb{R}^2$.
	The commutator estimate leading to the second inequality is a result of~\cite[Theorem 1.2]{L}, while the final step follows from the embeddings $H^s\subset L^{\frac{2}{1-s}}$ and $H^2\subset W^{1+s,\frac{2}{s}}$, for any $0<s<1$. Using~\eqref{eq:u^m in L^1+delta,L^2+veps} and~\eqref{eq:existence of characteristics}, the convergence of the approximate initial data, and the continuity of the extension operator, we have
	\begin{equation}
		\Vert \overline{u}^m \Vert_{L^1 W^{1,\infty}} , \Vert \overline{u}^m \Vert_{L^1 H^2} \le C(T,\Vert u_0^m \Vert_{H^s}, \Vert \overline{\rho_0}^m \Vert_{H^s}) \le C(T,\Vert u_0 \Vert_{H^s}, \Vert \rho_0 \Vert_{H^s}) .
	\end{equation}
	As a result, a straightforward application of G\"onwall's inequality gives us
	\begin{equation}
		\Vert \overline{\rho}^m(t) \Vert_{H^s} \le e^{C(T,\Vert u_0 \Vert_{H^s}, \Vert \overline{\rho_0} \Vert_{H^s})} \Vert \rho_0 \Vert_{H^s} .
	\end{equation}
	After establishing the strong continuity in time just as in Section~\ref{sec:time continuity of density}, we then proceed to extract a weak-* subsequence, and restrict the resulting solution $\overline{\rho}$ to $\Omega$, which yields the sought-after solution $\rho\in CH^s$. 
		
	\subsection{Higher fractional regularity} \label{sec:higher fractional regularity}
	We now deal with $(u_0,\rho_0) \in D(A^{\frac{s}{2}}) \times H^s$ when $1<s<2$. Similarly to Section~\ref{sec:lower fractional regularity}, the velocity bounds can be achieved using only $\rho_0\in H^{\lfloor s \rfloor} = H^1$. Owing to the regularity of the initial density, we have from Section~\ref{sec:gradient of the density} that $\rho^m \in L^{\infty} H^1$. We test~\eqref{eq:Galerkin approximation}$_1$ with $A^s u^m$ and apply Young's inequality to arrive at
	\begin{equation} \label{eq:testing momentum with A^s u^m higher fractional}
			\frac{d}{dt}\lVert A^\frac{s}{2}u^m \rVert_{L^2}^2 + \lVert A^\frac{s+1}{2}u^m \rVert_{L^2}^2
			\le
			C\lVert A^{\frac{s-1}{2}} P_m(u^m\cdot\nabla u^m)\rVert_{L^2}^2 + 
			C\lVert A^{\frac{s-1}{2}} P_m(\rho^m e_2)\rVert_{L^2}^2 .
	\end{equation}
	Once again, we have $\frac{s-1}{2}>-\frac{1}{2}$, so that the norm equivalence in~\cite{GuS} applies. Using that, Lebesgue interpolation, Sobolev embedding, and Lemma~\ref{lem:Stokes embedding}, we can simplify the first term on the RHS of~\eqref{eq:testing momentum with A^s u^m higher fractional} as
	\begin{equation*} \label{eq:higher fractional regularity velocity term}
		\begin{aligned}
			\lVert A^{\frac{s-1}{2}} P_m(u^m\cdot\nabla u^m)\rVert_{L^2} 
			&\lesssim 
			\lVert u^m\cdot\nabla u^m \rVert_{H^{s-1}} 
			\lesssim 
			\lVert u^m\cdot\nabla u^m \rVert_{W^{1,\frac{2}{3-s}}} \\
			&\lesssim
			\lVert (\nabla u^m)\cdot\nabla u^m \rVert_{L^{\frac{2}{3-s}}} + \lVert u^m\cdot \nabla\nabla u^m \rVert_{L^{\frac{2}{3-s}}} \\
			&\lesssim 
			\lVert \nabla u^m \rVert_{L^2} \lVert \nabla u^m \rVert_{L^{\frac{2}{2-s}}} + \lVert u^m \rVert_{L^{\frac{2}{2-s}}} \lVert D^2 u^m \rVert_{L^2} \\
			&\lesssim
			\lVert u^m \rVert_{H^1}^2 + \lVert u^m \rVert_{H^1} \lVert A u^m \rVert_{L^2} .
		\end{aligned}
	\end{equation*}
	Similarly,
	\begin{equation*} \label{eq:higher fractional regularity density term}
		\begin{aligned}
			\lVert A^{\frac{s-1}{2}} P_m(\rho^me_2)\rVert_{L^2} 
			\lesssim 
			\lVert A^{\frac{1}{2}} P_m(\rho^me_2)\rVert_{L^2}
            &\lesssim
            \lVert \rho^m \rVert_{H^1} \\
			&\le
			C\left(T, \lVert \rho_0^m \rVert_{H^1}, \lVert u_0^m \rVert_{H^s} \right)
			\le 
			C\left(T, \lVert \rho_0 \rVert_{H^1}, \lVert u_0 \rVert_{H^s} \right) .
		\end{aligned}
	\end{equation*}
	Using these estimates, along with~\eqref{eq:energy bound} and~\eqref{eq:higher order energy bounds}, we arrive at the desired bound by subjecting~\eqref{eq:testing momentum with A^s u^m higher fractional} to Gr\"onwall's inequality. Namely,
	\begin{equation} \label{eq:H^s energy bound higher fractional}
		\Vert{u^m}\Vert_{L^{\infty}H^s} + \lVert{u^m}\rVert_{L^2 H^{s+1}} 
		\le 
		C(T,\lVert u_0\rVert_{H^s},\lVert \rho_0\rVert_{H^1}) .
	\end{equation}
    Extracting appropriate subsequences, and observing that $(u_t,\rho_t)\in L^2 D(A^{\frac{s-1}{2}}) \times L^2 L^2$, we arrive at ${u\in C D(A^\frac{s}{2}) \cap L^2 D(A^{\frac{s+1}{2}})}$, while $\rho\in C H^1$. Section~\ref{sec:uniqueness of solutions} guarantees the uniqueness of the solution. This proves the case where $\rho_0\in H^{\lfloor s \rfloor} = H^1$. 
		
	Next, we use the solution $u$ from above and derive $H^s$ estimates for the density. We mirror the arguments in Section~\ref{sec:lower fractional regularity} to arrive at
	\begin{equation} \label{eq:testing density equation with high derivative}
		\frac{d}{dt}\Vert \Lambda^s \overline{\rho}^m \Vert_{L^2} \le \Vert \Lambda^s(\overline{u}^m\cdot\nabla\overline{\rho}^m) - \overline{u}^m\cdot\nabla\Lambda^s\overline{\rho}^m \Vert_{L^2} \lesssim \Vert \overline{u}^m \Vert_{H^{s+1}} \Vert \overline{\rho}^m \Vert_{H^s} .
	\end{equation}
	The last inequality is due to a commutator estimate proved in in~\cite[Theorem 1.2]{FMRR}. Using~\eqref{eq:H^s energy bound higher fractional} and Gr\"onwall's inequality, it is easy to see that 
	\begin{equation} \label{eq:H^s extended density bound higher fractional}
		\Vert \overline{\rho}^m(t) \Vert_{H^s} \le \Vert \overline{\rho_0}^m \Vert_{H^s} e^{C(T,\lVert u_0\rVert_{H^s},\lVert \rho_0\rVert_{H^1})} ,
	\end{equation}
	concluding the a~priori estimates for the higher fractional regularity. Time-continuity arguments, weak limits, and restriction of the solution to $\Omega$ then conclude the proof of higher fractional regularity. 
    \qed

\section{Asymptotic Properties --- Proof of Theorem~\ref{T05}} \label{sec:proof of asymptotic properties}
    
    We take the time derivative of~\eqref{eq:Leray projected momentum equation} and test it by $u_t$ obtaining
    \begin{equation} \label{eq:time derivative energy estimate}
        \frac{1}{2}\frac{d}{dt}\Vert u_t\Vert_{L^2}^2
        +\Vert \nabla u_t\Vert_{L^2}^2
        =
        -(u_t \cdot \nabla u, u_t) + (\rho_t e_2, u_t) .	
    \end{equation}
    For the first term on the right hand side, we employ Ladyzhenskaya's inequality and~\eqref{eq:energy bound} to write
    \begin{equation} \label{eq:time derivative energy estimate RHS 1}
        \begin{aligned}
            -(u_t \cdot \nabla u, u_t)
            =(u_t \cdot \nabla u_t, u)
            &\lec
            \Vert \nabla u_t\Vert_{L^2}\Vert u_t\Vert_{L^4}\Vert u\Vert_{L^4}
            \lec
            \Vert \nabla u_t\Vert_{L^2}^{\frac{3}{2}} \Vert u_t\Vert_{L^2}^{\frac{1}{2}} \Vert u\Vert_{L^2}^{\frac{1}{2}} \Vert \nabla u\Vert_{L^2}^{\frac{1}{2}}
            \\&\le
            \epsilon\Vert \nabla u_t\Vert_{L^2}^2 
            + C_{\epsilon}\Vert u_t\Vert_{L^2}^2 \Vert \nabla u\Vert_{L^2}^2 .
        \end{aligned}
    \end{equation}
    For the term involving $\rho_t$ in \eqref{eq:time derivative energy estimate}, we use the density equation to get
    \begin{equation} \label{eq:time derivative energy estimate RHS 2}
        \begin{aligned}
            (\rho_t e_2, u_t)
            =
            -(\nabla\cdot(\rho u)e_2, u_t)
            &=
            (\rho u, \nabla(u_t \cdot e_2)) \\
            &\le
            \Vert \nabla u_t\Vert_{L^2} \Vert \rho \Vert_{L^4} \Vert u\Vert_{L^4} 
            \le
            \epsilon\Vert \nabla u_t\Vert_{L^2}^2
            +
            C_{\epsilon} \Vert \nabla u\Vert_{L^2}^2 .
        \end{aligned}
    \end{equation}
    In the last step, we used Sobolev embedding,~\eqref{eq:density L^p norm}, and the Poincar\'e inequality. Now, absorbing the factors of $\Vert \nabla u_t\Vert_{L^2}^2$ and using Gronwall's inequality, we have
    \begin{equation} \label{eq:time derivative Gronwall's}
        \Vert u_t(t)\Vert_{L^2}^2
        +
        \int_0^t \Vert \nabla u_t\Vert_{L^2}^2 \,ds
        \lec
        \left( \Vert u_t(0)\Vert_{L^2}^2 + \int_0^t \Vert \nabla u\Vert_{L^2}^2 \,ds\right)
        \exp\left(
        \int_0^t \Vert \nabla u\Vert_{L^2}^2 \,ds
        \right).
    \end{equation}
    Thanks to~\eqref{eq:solution regularity for k=2} and~\eqref{eq:energy bound}, the right hand side of the above inequality is bounded uniformly in time. Therefore, we conclude that
    $\Vert \nabla u_t\Vert_{L^2} \to 0$ as $t \to \infty$. This implies, from the velocity equation and~\eqref{eq:T00 auxiliary estimates  for velocity}, that
    \begin{equation} \label{eq:Au - P(rho e_2) goes to 0 in V}
        \Vert Au - \mathbb{P}(\rho e_2)\Vert_{V} 
        \lec 
        \Vert \nabla u_t\Vert_{L^2}+\Vert u \cdot \nabla u\Vert_{H^1} \xrightarrow[]{t\to\infty} 0 .
    \end{equation}
    Next, recalling \eqref{eq:projection1}, we apply $Q$ to \eqref{eq:boussinesq equations}$_1$ obtaining
    \begin{equation} \label{eq:Q on momentum equation}
       Q(\nabla p) - Q(\rho e_2)= -Q(u\cdot \nabla u) ,
    \end{equation}
    since $\div(u_t + \Delta u) = 0$. Using the continuity of $Q$ on $H^1$ and~\eqref{eq:T00 auxiliary estimates for velocity}, we conclude that the right-hand side of \eqref{eq:Q on momentum equation} converges strongly to $0$ in $H^1 \cap H'$,
    proving the first part of Theorem~\ref{T05}. We now proceed to the second part, which concerns an asymptotic steady state for the density.
    
    \textit{(i) $\Rightarrow$ (ii)}: Owing to the continuity of $\mathbb{P}$ on $L^2$, and because $\mathbb{P}(\rho e_2)$ weakly converges to $0$ in $L^2$, we have $\lim_{t \to \infty} \Vert \mathbb{P}(\rho e_2)\Vert_{L^2} = 0$. Since $\mathbb{P}$ and $I - \mathbb{P}$ are orthogonal, we conclude that $(I-\mathbb{P})(\rho e_2)$ converges strongly to $\bar{\rho}e_2$ in $L^2$ as $t \to \infty$.

    \textit{(ii) $\Rightarrow$ (i)}: Observe that the conservation of the $L^2$ norm of $\rho e_2$ and the orthogonality
    of $\mathbb{P}$ and $I - \mathbb{P}$ imply that
    \begin{align}
        \Vert \rho_0 e_2 \Vert_{L^2}^2 
        = \Vert \rho e_2\Vert_{L^2}^2
        = \Vert \mathbb{P}(\rho e_2)\Vert_{L^2}^2
        + \Vert (I - \mathbb{P})(\rho e_2)\Vert_{L^2}^2.
    \end{align}
    Therefore, taking the limit as $t \to \infty$, and using the assumption that $\Vert \bar{\rho} \Vert_{L^2} = \Vert \rho_0 \Vert_{L^2}$,
    we conclude that $\mathbb{P}(\rho e_2)$ converges strongly to $0$ in $L^2$. 
    Since $\rho e_2 = \mathbb{P}(\rho e_2) + (I - \mathbb{P})(\rho e_2)$,
    we obtain $\lim_{t \to \infty} \rho e_2 = \bar{\rho}e_2$.

    Finally, we assume $(i)$ and establish \eqref{eq:asy2}. From the aforementioned arguments, we necessarily have
    $\lim_{t \to \infty} \mathbb{P}(\rho e_2) = 0$ in $L^2$. Thus, combining this with \eqref{eq:T00 auxiliary estimates  for velocity}$_2$, we obtain $\lim_{t \to \infty} \lVert Au \rVert_{L^2} = 0$. Furthermore, since $\Vert \Delta u\Vert_{L^2} \lec \Vert Au\Vert_{L^2}$,~\eqref{eq:boussinesq equations}$_1$ yields
    \begin{equation}
        \Vert \nabla p - \rho e_2 \Vert_{L^2} \lesssim \Vert u_t\Vert_{L^2} + \Vert \Delta u\Vert_{L^2} + \Vert u\cdot\nabla u\Vert_{L^2} \xrightarrow[t\to\infty]{} 0 .
    \end{equation}
    From here, we arrive at~\eqref{eq:asy2} since $\rho \to \overline{\rho}$ as $t\to\infty$.
    \qed






\section*{Acknowledgments}
    MSA was supported in part by the NSF grant DMS-2205493. Part of this material is based upon work supported by the NSF under Grant No.~DMS-1928930 while MSA was in residence at the SLMATH in Berkeley, California, during the summer of 2023. Both authors are grateful to Juhi Jang and Igor Kukavica for helpful discussions.


\begin{thebibliography}{[DWZZ]}
	\small
	\bibitem[ACW]{ACW} 
	D.~Adhikari, C.~Cao, and J.~Wu, 
	\emph{Global regularity results for the 2{D} {B}oussinesq equations with vertical dissipation}, 
	J.~Differential Equations~\textbf{251} (2011), no.~6, 1637--1655. 
	
	\bibitem[ACS..]{ACSetal} 
	D.~Adhikari, C.~Cao, H.~Shang, J.~Wu, X.~Xu, and Z.~Ye,
	\emph{Global regularity results for the 2{D} {B}oussinesq equations with partial dissipation}, 
	J.~Differential Equations~\textbf{260} (2016), no.~2, 1893--1917. 
	
	\bibitem[AKZ]{AKZ}
	M.S.~Aydin, I.~Kukavica, and M.~Ziane,
	\emph{On Asymptotic Properties of the Boussinesq Equations},
	arXiv:2304.00481.
	
	\bibitem[BS]{BS}
	L.C.~Berselli and S.~Spirito, \emph{On the {B}oussinesq system: 
		regularity criteria and singular limits},  
	Methods Appl.\ Anal.~\textbf{18} (2011), no.~4, 391--416. 
	
	\bibitem[BFL]{BFL} 
	A.~Biswas, C.~Foias, and A.~Larios, 
	\emph{On the attractor for the semi-dissipative {B}oussinesq equations}, 
	Ann.\ Inst.\ H.~Poincar\'{e} Anal.\ Non Lin\'{e}aire~\textbf{34} (2017), no.~2, 381--405. 
	
	\bibitem[BrS]{BrS} 
	L.~Brandolese and M.E.~Schonbek, 
	\emph{Large time decay and growth for solutions of a viscous {B}oussinesq system}, 
	Trans.\ Amer.\ Math.\ Soc.\ ~\textbf{364} (2012), no.~10, 5057--5090. 
	
	\bibitem[C]{C}
	D.~Chae, \emph{Global regularity for the 2{D} {B}oussinesq equations with partial viscosity terms}, 
	Adv.\ Math.~\textbf{203} (2006), no.~2, 497--513.
	
	\bibitem[CD]{CD}
	J.R.~Cannon and E.~DiBenedetto, 
	\emph{The initial value problem for the {B}oussinesq equations with data in {$L^{p}$}}, 
	Approximation methods for {N}avier-{S}tokes problems 
	({P}roc.\ {S}ympos., {U}niv.\ {P}aderborn, {P}aderborn, 1979), 
	Lecture Notes in Math., vol.~771, Springer, Berlin, 1980, pp.~129--144. 
	
	
	\bibitem[CF]{CF}
	P.~Constantin and C.~Foias, 
	\emph{Navier-{S}tokes equations}, Chicago Lectures in Mathematics, University of Chicago Press, Chicago, IL, 1988.
	
	\bibitem[CG]{CG}
	M.~Chen and O.~Goubet, 
	\emph{Long-time asymptotic behavior of two-dimensional dissipative {B}oussinesq systems}, 
	Discrete Contin.\ Dyn.\ Syst.\ Ser.~S~\textbf{2} (2009), no.~1, 37--53. 
	
	\bibitem[CH1]{CH1} 
	J.~Chen and T.Y.~Hou, 
	\emph{Finite time blowup of 2{D} {B}oussinesq and 3{D} {E}uler equations with {$C^{1,\alpha}$} velocity and boundary},
	Comm.\ Math.\ Phys.~\textbf{383} (2021), no.~3, 1559--1667. 
	
	\bibitem[CH2]{CH2} 
	J.~Chen and T.Y.~Hou, 
	\emph{Stable nearly self-similar blowup of the 2D Boussinesq
		and 3D Euler Equations with smooth data II: Rigorous Numerics},
	arXiv:2305.05660, 2023. 
	
	\bibitem[CKN]{CKN}
	D.~Chae, S.-K.~Kim and H.-S.~Nam, 
	\emph{Local existence and blow-up criterion of H\"older-continuous solutions of the {B}oussinesq equations}, 
	Nagoya Math. J. ~\textbf{155} (1999), 55--80. 
	
	\bibitem[CN]{CN}
	D.~Chae and H.-S.~Nam, 
	\emph{Local existence and blow-up criterion for the {B}oussinesq equations}, 
	Proc.\ Roy.\ Soc.\ Edinburgh Sect.~A~\textbf{127} (1997), no.~5, 935--946. 
	
	\bibitem[CW]{CW}
	C.~Cao and J.~Wu, 
	\emph{Global regularity for the two-dimensional anisotropic {B}oussinesq equations with vertical dissipation}, 
	Arch.\ Ration.\ Mech.\ Anal.~\textbf{208} (2013), no.~3, 985--1004. 
	
	
	\bibitem[DM]{DM}
	R.~Danchin and P.B.~Mucha, 
	\emph{The Incompressible Navier-Stokes Equations in Vacuum},
	Comm.\ Pure Appl. Math.~\textbf{72} (2019), no.~7, 1351--1385. 
	
	
	\bibitem[DP]{DP}
	R.~Danchin and M.~Paicu, 
	\emph{Les th\'eor\`emes de {L}eray et de {F}ujita-{K}ato pour le syst\`eme de {B}oussinesq partiellement visqueux},
	Bull.\ Soc.\ Math.\ France~\textbf{136} (2008), no.~2, 261--309. 
	
	\bibitem[DWZZ]{DWZZ} 
	C.R.~Doering, J.~Wu, K.~Zhao, and X.~Zheng, 
	\emph{Long time behavior of the two-dimensional {B}oussinesq equations without buoyancy diffusion}, 
	Phys.~D~\textbf{376/377} (2018), 144--159. 
	
	
	\bibitem[EJ]{EJ} 
	T.M.~Elgindi and I.-J.~Jeong, 
	\emph{Finite-time singularity formation for strong solutions to the {B}oussinesq system}, 
	Ann.\ PDE~\textbf{6}~(2020), no.~1, Paper No. 5, 50. 
	
	\bibitem[FHR]{FHR}
	C.L.~Fefferman, K.W.~Hajduk and J.C.~Robinson, 
	\emph{Simultaneous approximation in Lebesgue and Sobolev norms via eigenspaces}, arXiv:1904.03337v2.
	
	\bibitem[FMRR]{FMRR}
	C.L.~Fefferman, D.S.~McCormick, J.C.~Robinson and J.L.~Rodrigo,
	\emph{Higher order commutator estimates and local existence for the non-resistive MHD equations and related models},
	J.~Funct.~Anal. \textbf{267}, 4 (2014) 1035--1056.
	
	\bibitem[GO]{GO}
	L.~Grafakos and S.~Oh 
	\emph{The Kato-Ponce Inequality},
	Comm.\ Partial Differential Equations \textbf{39} (2014) no.~6, 1128--1157
	
	\bibitem[GS]{GS}
	Y.~Giga and H.~Sohr, 
	\emph{Abstract $L^p$ estimates for the Cauchy problem with applications to the Navier-Stokes equations in exterior domains},  
	J.~Funct.~Anal. \textbf{102}, (1991) 72--94.
	
	\bibitem[GuS]{GuS}
	J.~Guermond and A.~Salgado, 
	\emph{A note on the Stokes operator and its powers},  
	J.~Appl.~Math.~Comput. \textbf{36}, (2011) 241--250.
	
	\bibitem[H]{H}
	L.~He,
	\emph{Smoothing estimates of 2d incompressible Navier–Stokes equations in bounded domains with applications},
	J.~Funct.~Anal. \textbf{262} (2012), no.~7, 3430--3464.
	
	\bibitem[HK1]{HK1}
	T.~Hmidi and S.~Keraani, 
	\emph{On the global well-posedness of the two-dimensional {B}oussinesq system with a zero diffusivity}, 
	Adv.\ Differential Equations~\textbf{12} (2007), no.~4, 461--480. 
	
	\bibitem[HK2]{HK2}
	T.~Hmidi and S.~Keraani, 
	\emph{On the global well-posedness of the {B}oussinesq system with zero viscosity}, 
	Indiana Univ.\ Math.~J.~\textbf{58} (2009), no.~4, 1591--1618. 
	
	\bibitem[HKR]{HKR}
	T.~Hmidi, S.~Keraani, and F.~Rousset, 
	\emph{Global well-posedness for {E}uler-{B}oussinesq system with critical dissipation}, 
	Comm.\ Partial Differential Equations~\textbf{36} (2011), no.~3, 420--445. 
	
	\bibitem[HL]{HL}
	T.Y.~Hou and C.~Li, 
	\emph{Global well-posedness of the viscous {B}oussinesq equations}, 
	Discrete Contin.\ Dyn.\ Syst.~\textbf{12} (2005), no.~1, 1--12. 
	
	\bibitem[HKZ1]{HKZ1}
	W.~Hu, I.~Kukavica, and M.~Ziane, 
	\emph{On the regularity for the {B}oussinesq equations in a bounded domain}, 
	J.~Math.\ Phys.~\textbf{54} (2013), no.~8, 081507, 10. 
	
	\bibitem[HKZ2]{HKZ2} 
	W.~Hu, I.~Kukavica, and M.~Ziane, 
	\emph{Persistence of regularity for the viscous {B}oussinesq equations with zero diffusivity}, 
	Asymptot.\ Anal.~\textbf{91} (2015), no.~2, 111--124. 
	
	\bibitem[HW]{HW} 
	W.~Hu and J.~Wu,
	\emph{An approximating approach for boundary control of optimal mixing via {N}avier-{S}tokes flows},
	J.~Differential Equations~\textbf{267} (2019), no.~10, 5809--5850.
	
	\bibitem[HWW+]{HWW+} 
	W.~Hu, Y.~Wang, J.~Wu, B.~Xiao, and J.~Yuan,
	\emph{Partially dissipative 2{D} {B}oussinesq equations with {N}avier type boundary conditions},
	Phys.~D~\textbf{376/377} (2018), 39--48.
	
	\bibitem[HS]{HS} 
	F.~Hadadifard and A.~Stefanov, 
	\emph{On the global regularity of the 2D critical Boussinesq system with $\alpha >2/3$}, 
	Comm.\ Math.\ Sci.~\textbf{15} (2017), no.~5, 1325--1351.
	
	\bibitem[J]{J}
	O.~Jarrin,
	\emph{On the existence, regularity and uniqueness of $L^p$-solutions to the steady-state 3D Boussinesq system in the whole space},
	arXiv:2307.11670.
	
	\bibitem[JK]{JK}
	J.~Jang and J.~Kim,
	\emph{Asymptotic stability and sharp decay rates to the linearly stratified
		Boussinesq equations in horizontally periodic strip domain},
	arXiv:2211.13404.
	
	\bibitem[JMWZ]{JMWZ} 
	Q.~Jiu, C.~Miao, J.~Wu, and Z.~Zhang, 
	\emph{The two-dimensional incompressible {B}oussinesq equations with general critical dissipation}, 
	SIAM J.~Math.\ Anal.~\textbf{46} (2014), no.~5, 3426--3454.
	
	\bibitem[Ju]{Ju} 
	N.~Ju, 
	\emph{Global regularity and long-time behavior of the solutions to the 2{D} {B}oussinesq equations without diffusivity in a bounded domain}, 
	J.~Math.\ Fluid Mech.~\textbf{19} (2017), no.~1, 105--121. 
	
	\bibitem[KMZ]{KMZ}
	I.~Kukavica, D.~Massatt and M.~Ziane,
	\emph{Asymptotic properties of the Boussinesq Equations with Dirichlet Boundary Conditions},
	Discrete Contin.\ Dyn.\ Syst. (to appear), arXiv:2109.14672.
	
	\bibitem[KLN]{KLN}
	K.~Kang, J.~Lee and D.D.~Nguyen,
	\emph{Global well-posedness and stability of the 2D Boussinesq
		equations with partial dissipation near a hydrostatic
		equilibrium},
	arXiv:2306.08286.
	
	\bibitem[KMPT]{KMPT}
	T.~Kato, M.~Mitrea, G.~Ponce, and M.~Taylor,
	\emph{Extension and representation of divergence-free vector fields on bounded domains},
	Math.\ Res.\ Lett.~\textbf{7} (2000), 643--650.
	
	\bibitem[KTW]{KTW}
	J.P.~Kelliher, R.~Temam, and X.~Wang, 
	\emph{Boundary layer associated with the {D}arcy-{B}rinkman-{B}oussinesq model for convection in porous media}, Phys.\ D~\textbf{240} (2011), no.~7, 619--628. 
	
	\bibitem[KW1]{KW1} 
	I.~Kukavica and W.~Wang, 
	\emph{Global {S}obolev persistence for the fractional {B}oussinesq equations with zero diffusivity}, 
	Pure Appl.\ Funct.\ Anal.~\textbf{5} (2020), no.~1, 27--45. 
	
	\bibitem[KW2]{KW2} 
	I.~Kukavica and W.~Wang, 
	\emph{Long time behavior of solutions to the 2{D} {B}oussinesq equations with zero diffusivity}, 
	J.~Dynam.\ Differential Equations~\textbf{32} (2020), no.~4, 2061--2077. 
	
	\bibitem[KWZ]{KWZ} 
	I.~Kukavica, F.~Wang and M.~Ziane, 
	\emph{Persistence of regularity for solutions of the Boussinesq equations in Sobolev spaces}, 
	Adv.\ Differential Equations~\textbf{21} (2016), no.~1/2, 85--108.
	
	\bibitem[KPY]{KPY}
	A.~Kiselev, J.~Park and Y.~Yao,
	\emph{Small Scale Formation for th 2{D} {B}oussinesq Equation},
	arXiv:2211.05070.
	
	\bibitem[LLT]{LLT}
	A.~Larios, E.~Lunasin, and E.S.~Titi, 
	\emph{Global well-posedness for the 2{D} {B}oussinesq system with anisotropic viscosity and without heat diffusion}, 
	J.~Differential Equations~\textbf{255} (2013), no.~9, 2636--2654.
	
	\bibitem[LPZ]{LPZ}
	M.-J.~Lai, R.~Pan, and K.~Zhao, 
	\emph{Initial boundary value problem for two-dimensional viscous {B}oussinesq equations}, 
	Arch.\ Ration.\ Mech.\ Anal.~\textbf{199} (2011), no.~3, 739--760. 
	
	\bibitem[L]{L}
	D.~Li, 
	\emph{On Kato–Ponce and fractional Leibniz}, 
	Rev.\ Mat.\ Iberoam.~\textbf{35} (2019), no.~1, 23--100. 
	
	
	\bibitem[QDY]{QDY}
	H.~Qiu, Y.~Du and Z.~Yao, 
	\emph{Local existence and blow-up criterion for the generalized {B}oussinesq equations in Besov spaces}, 
	Math.\ Meth.\ Appl.\ Sci.\ \textbf{36} (2013), 86–-98.
	
	
	\bibitem[RRS]{RRS}
	J.~Robinson, J.~Rodrigo, and W.~Sadowski, 
	\emph{The Three-Dimensional Navier–Stokes Equations: Classical Theory}, 
	Cambridge University Press, Cambridge, United Kingdom, 2016.
	
	\bibitem[S]{S}
	J.-I.~Saito, 
	\emph{Boussinesq equations in thin spherical domains}, 
	Kyushu J.\ Math.~\textbf{59} (2005), no.~1, 443--465.
	
	\bibitem[SW]{SW} 
	A.~Stefanov and J.~Wu, 
	\emph{A global regularity result for the 2{D} {B}oussinesq equations with critical dissipation}, 
	J.~Anal.\ Math.~\textbf{137} (2019), no.~1, 269--290. 
	
	\bibitem[SvW]{SvW} 
	H.~Sohr and W.~von Wahl, 
	\emph{On the regularity of the pressure of weak solutions of Navier-Stokes equations}, 
	Arch.\ Math.~\textbf{46} (1986), 428--439. 
 
	\bibitem[T1]{T1} 
	R.~Temam, 
	\emph{Navier-{S}tokes equations}, 
	AMS Chelsea Publishing, Providence, RI, 2001, Theory and numerical analysis, Reprint of the 1984 edition. 
	
	\bibitem[T2]{T2}
	R.~Temam, 
	\emph{Navier-{S}tokes equations and nonlinear functional analysis}, second ed., 
	CBMS-NSF Regional Conference Series in Applied Mathematics, vol.~66, Society for Industrial and Applied Mathematics (SIAM), Philadelphia, PA, 1995. 
	
	\bibitem[T3]{T3} 
	R.~Temam, 
	\emph{Behaviour at time {$t=0$} of the solutions of semilinear evolution equations}, 
	J.~Differential Equations~\textbf{43} (1982), no.~1, 73--92. 

    \bibitem[W]{W} 
	W.~Wang, 
	\emph{On the global regularity for a 3{D} {B}oussinesq model without thermal diffusion}, 
	Z. Angew. Math. Phys.~\textbf{70} (2019), no.~6, Paper No. 174, 6.
 
    \bibitem[WY]{WY} 
	W.~Wang and H.Yue, 
	\emph{Almost sure existence of global weak solutions to the {B}oussinesq equations}, 
	Dyn. Partial Differ. Equ.~\textbf{17} (2020), no.~2, 165--183.
    
	\bibitem[WZ]{WZ} 
	C.~Wang and Z.~Zhang, 
	\emph{Global well-posedness for the 2-D {B}oussinesq system with the temperature-dependent viscosity and thermal diffusivity}, 
	Advances in Mathematics~\textbf{228} (2011), 43--62.
	
\end{thebibliography}
\end{document}